\documentclass[sn-mathphys,Numbered]{sn-jnl}

\usepackage{graphicx}%
\usepackage{amsmath,amssymb,amsfonts}%
\usepackage{amsthm}%
\usepackage{mathrsfs}%

\usepackage{mathtools}
\usepackage{xcolor}%

\usepackage{manyfoot}%

\usepackage{pgfplots}
\usepgfplotslibrary{fillbetween}
\usetikzlibrary{patterns}
\usepackage{scalerel}%
\pgfplotsset{compat=1.18} 

\usepackage[nameinlink]{cleveref}
\usepackage[colorinlistoftodos]{todonotes}
\newcommand{\x}{\boldsymbol x}

\newcommand{\R}{\mathbb R}
\newcommand{\N}{\mathbb N}

\newcommand{\eps}{\varepsilon}

\newcommand{\dd}{\ensuremath{\,\mathrm{d}}}

\DeclareMathOperator*{\esssup}{ess\,sup}
\DeclareMathOperator*{\essinf}{ess\,inf}

\renewcommand{\epsilon}{\varepsilon}

\crefname{ineq}{Inequality}{Inequalities}
\creflabelformat{ineq}{#2{(#1)}#3}
\crefname{assumption}{Assumption}{Assumptions}

\crefname{remark}{Rem.}{Remarks}
\crefname{lemma}{Lem.}{Lemmata}
\crefname{proposition}{Prop.}{Propositions}
\crefname{theorem}{Thm.}{Theorems}
\crefname{assumption}{Asm.}{Assumptions}
\crefname{definition}{Defn.}{Definitions}
\crefname{example}{Ex.}{Examples}
\crefname{section}{Sec.}{Sections}
\crefname{figure}{Fig.}{Figures}
\crefname{equation}{Eq.}{Eqs.}

\usepackage{color}
\definecolor{mygreen}{RGB}{28,172,0} 
\definecolor{mylilas}{RGB}{170,55,241}
\definecolor{faublue}{RGB}{0,56,101}
\definecolor{fauorange}{RGB}{201,147,19}
\definecolor{faured}{RGB}{141,20,41}
\definecolor{faucyan}{RGB}{0,177,235}
\definecolor{faugreen}{RGB}{0,122,93}
\definecolor{fau-nat-green}{RGB}{0,122,93}
\definecolor{faugray}{RGB}{152,164,174}

\definecolor{matlabred}{RGB}{162, 20, 47}
\definecolor{matlabgreen}{RGB}{119, 173, 48}
\definecolor{matlabcyan}{RGB}{77, 191, 239}
\definecolor{matlabblue}{RGB}{0, 114, 190}
\definecolor{matlabyellow}{RGB}{238, 178, 32}
\definecolor{matlaborange}{RGB}{218, 83, 25}
\definecolor{matlabpurple}{RGB}{126, 47, 142}

\definecolor{matlabred1}{RGB}{121.5000, 43.5000, 82.7500}
\definecolor{matlabred2}{RGB}{81.0000, 67.0000, 118.5000}
\definecolor{matlabred3}{RGB}{40.5000, 90.5000, 154.2500}

\definecolor{matlabblue1}{RGB}{108, 61.3, 94.6}
\definecolor{matlabblue2}{RGB}{54, 102.6, 142.3}

\newtheorem{theorem}{Theorem}
\newtheorem{lemma}[theorem]{Lemma}%
\newtheorem{assumption}[theorem]{Assumption}%

\renewcommand*\descriptionlabel[1]{\hspace\labelsep\textbf {#1}}

\newtheorem{example}{Example}%
\newtheorem{remark}{Remark}%

\newtheorem{definition}{Definition}%

\begin{document}

\title[Obstacle problem for linear conservation laws]{The obstacle problem for linear scalar conservation laws with constant velocity}

\author[4,5]{\fnm{Paulo} \sur{Amorim}}\email{paulo@im.ufrj.br}
\author[1]{\fnm{Alexander} \sur{Keimer}}\email{alexander.keimer@fau.de}
\author[2,3]{\fnm{Lukas} \sur{Pflug}}\email{lukas.pflug@fau.de}
\author[2]{\fnm{Jakob} \sur{Rodestock}}\email{jakob.w.rodestock@fau.de}

\affil[1]{\orgdiv{Department Mathematics}, \orgname{Friedrich-Alexander Universität Erlangen-Nürnberg (FAU)}, \orgaddress{\street{Cauerstr. 11}, \city{Erlangen}, \postcode{91058}, \country{Germany}}}

\affil[2]{\orgdiv{Competence Center for Scientific Computing (FAU CSC)}, \orgname{Friedrich-Alexander Universität Erlangen-Nürnberg (FAU)}, \orgaddress{\street{Martensstr. 5a}, \city{Erlangen}, \postcode{91058}, \country{Germany}}}

\affil[3]{\orgdiv{Chair of Applied Mathematics, Continuous Optimization, Department Mathematics}, \orgname{Friedrich-Alexander Universität Erlangen-Nürnberg (FAU)}, \orgaddress{\street{Cauerstrasse 11}, \city{Erlangen}, \postcode{91058}, \country{Germany}}}

\affil[4]{\orgdiv{School of Applied Mathematics (FGV-EMAp)}, \orgname{Fundação Getúlio Vargas}, \orgaddress{\street{Praia de Botafogo 190}, \city{Rio de Janeiro}, \postcode{22250-900}, \country{Brazil}}}

\affil[5]{\orgdiv{Instituto de Matemática}, \orgname{Universidade Federal do Rio de Janeiro (UFRJ)}, \orgaddress{\street{ Av. Athos da Silveira Ramos 149}, \city{Rio de Janeiro}, \postcode{C.P. 68530, 21941-909}, \country{Brazil}}}


\abstract{In this contribution, we present a novel approach for solving the obstacle problem for (linear) conservation laws. Usually, given a conservation law with an initial datum, the solution is uniquely determined.  How to incorporate obstacles, i.e., inequality constraints on the solution so that the resulting solution is still ``physically reasonable'' and obeys the obstacle, is unclear. The proposed approach involves scaling down the velocity of the conservation law when the solution approaches the obstacle. We demonstrate that this leads to a reasonable solution and show that, when scaling down is performed in a discontinuous fashion, we still obtain a suitable velocity - and the solution satisfying a discontinuous conservation law. We illustrate the developed solution concept using numerical approximations.}




\maketitle


\section{Introduction}

\subsection{The obstacle problem in the literature}
Hyperbolic conservation laws involving constraints on the solution (usually called \emph{obstacles}) arise naturally in many applications. For instance, in traffic modeling, it is natural to impose that vehicle density cannot exceed a certain threshold. 
Bounds on the velocity can be envisioned, given by speed limits, or on the vehicular flux in a certain region \cite{1berthelin2008model,1andreianov,1andreianov2,1chalons,1berthelin,denitti2024pointwise,1colombo,1dymski,1garavello,1garavello2,Bayen2022}. Elsewhere, in multi-phase flows modeled by hyperbolic equations, constraints may appear in the solutions, leading to obstacles. For related problems involving hyperbolic settings,  \cite{1levi,1rodrigues,1rodrigues2,1goudon,1saldanha,1berthelin-bouchut,Rossi2020,FernandezReal2020} provide a few examples. In addition, pedestrian flows may involve congestion,  so population density should be limited; see \cite{1bellomo2022towards} for a recent survey. \\
More generally, in cases where the solution of a partial differential equation represents the density of a medium, e.g., a population, fluid, group of vehicles, or group of pedestrians, one may consider a variety of situations where a constraint on the density is imposed in a certain region of space or time. Therefore, since, in many cases, such dynamics are described by parabolic or hyperbolic evolution equations, a modeling mechanism capable of modifying an existing problem is useful to introduce density constraints in a meaningful way. This is especially interesting in cases where the solutions of the ``free'' problem (i.e.,~without constraint) typically violate the constraint; then,  additional mechanisms that enforce the obstacle can be defined while remaining consistent with the underlying model.

The existing literature on obstacle problems is already substantial, especially concerning parabolic or elliptic problems. Since this is not the focus of our work, we cite only a few selected contributions, which are either foundational/surveys/monographs \cite{2lions,2kinder,2rodriguesbook,2mignot,2brezis,2rudd,2korte} or publications involving interesting applications \cite{2rodrigues,chalub,2bogelein}. Most approaches, especially in the parabolic setting, state the obstacle problem in a variational setting, where connecting the mathematical tools used to proper physical interpretations is not easy. Namely, a typical method to obtain a solution to an obstacle problem is by \emph{penalization}. For instance, when a function $u$ that must remain below some obstacle $\phi$ is sought, the evolution equation is supplemented by some source terms acting to decrease $u$ on the set where $u>\phi$ in the setting of some approximation scheme. In general, though, such terms may have no ready interpretation in the model setting and often represent mathematical artifacts reminiscent of more abstract problems intended to find appropriate projections on convex sets \cite{2brezis,2lions}.

For hyperbolic conservation laws, in particular, the penalization approach has been applied by a few authors, e.g.,~\cite{1levi,1berthelin-bouchut,1amorim}. One drawback is that a penalization term is, in general, not consistent with the conservation (or evolution) of the total mass. However, a constraint on a point-wise density should not necessarily affect the evolution of the total mass. In view of this, in the setting of an obstacle problem for a conservation law preserving mass, authors in \cite{1amorim} introduced a Lagrange multiplier (effectively a source term) designed to counterbalance the effect of the penalization on the total mass. However, the resulting formulation relies on mass creation and destruction to enforce the obstacle constraint while preserving the total mass. Here, we propose a different approach, which we argue is more natural.

Indeed, when trying to enforce an obstacle-like condition, density variations should originate from rearrangement of individuals rather than the destruction/creation of mass, especially in the case of traffic or pedestrian models. Thus, we expect that changes in the local \emph{velocity} of the vehicles or individuals can be enough to adjust the density to satisfy the constraint. For a comparison of the two approaches, see~\cref{paulo}.
\begin{figure}
    \centering
    \includegraphics[scale=0.6,clip,trim=0 22 12 0 0]{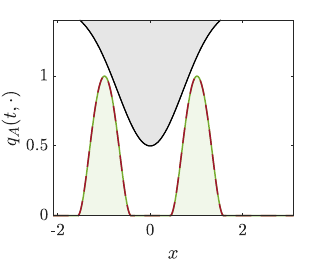}
    \includegraphics[scale=0.6,clip,trim=25 22 12 0 0]{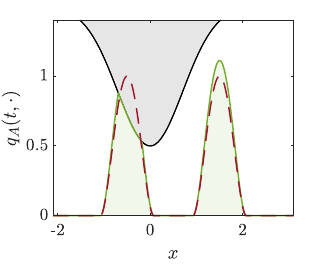}
    \includegraphics[scale=0.6,clip,trim=25 22 12 0 0]{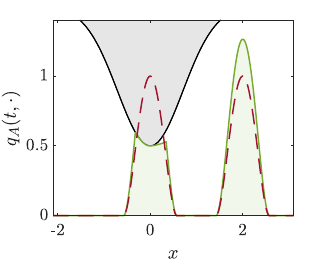}
    \includegraphics[scale=0.6,clip,trim=25 22 12 0 0]{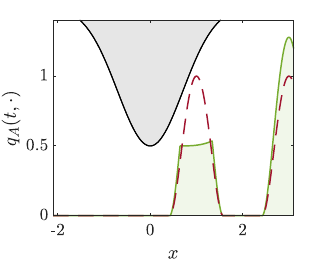}

    \includegraphics[scale=0.6,clip,trim=0 0 12 0 0]{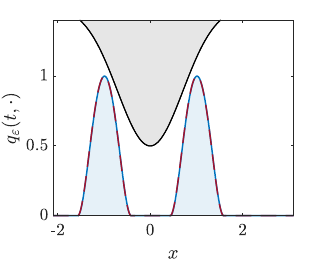}
    \includegraphics[scale=0.6,clip,trim=25 0 12 0 0]{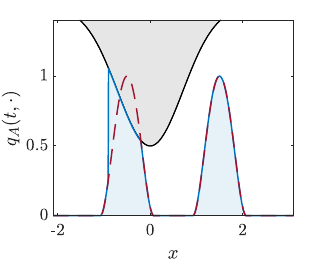}
    \includegraphics[scale=0.6,clip,trim=25 0 12 0 0]{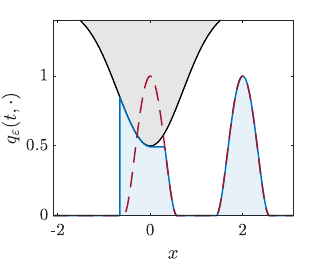}
    \includegraphics[scale=0.6,clip,trim=25 0 12 0 0]{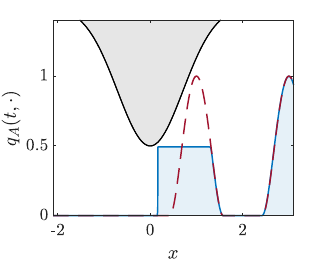}
\\
\caption{Solution \textcolor{matlabgreen}{$q_A$} proposed in~\cite{1amorim} with constant velocity of $1$ for times $t = 0, 0.5, 1, 2$ (top from left to right) and the \textcolor{matlabblue}{solution} suggested in~\cref{eq:conservation_law_smooth} for times $t = 0, 0.5, 1, 2$ (bottom from left to right) and $\eps = \tfrac{1}{1024}$ with identical initial datum $q_0$ and obstacle $o$ (black). The dashed line shows the function $q_0(\cdot-t)$.}
    \label{paulo}
\end{figure}

With these remarks in mind, we propose a new formulation for the obstacle problem for a one-dimensional hyperbolic conservation law. In this work, we consider the case of a linear flux only. As shown below, the linear case allows for a cleaner exposition, focusing on the main innovations of our approach.

\subsection{Problem statement and outline}

The obstacle problems for linear conservation laws with constant coefficients reads
\begin{align*}
    \partial_{t}q +\partial_{x}q&=0,\quad q(0,\cdot)\equiv q_{0}\ \text{ such that }\  q(t,x)\leq o(x)\ \forall (t,x)\in(0,T)\times\R
\end{align*}
for an obstacle \(o:\R\rightarrow\R\). As evident, the proposed system is over-determined in the sense that the Cauchy problem admits a unique solution without the obstacle. As a result, how to incorporate the obstacle into the solution is not straightforward when the following holds:
\begin{itemize}
    \item mass is conserved even when hitting the obstacle;
    \item the solution still satisfies a semi-group property in time;
    \item the solution behaves physically reasonably in the sense that mass is not instantaneously transported across space.
\end{itemize}
Having this in mind, one may need to adjust the velocity of the conservation law accordingly.
Following this approach, we want the velocity of the conservation law to decrease when the solution approaches the obstacle. Thus, for a smoothed version of the Heaviside function \(V_{\eps},\ \eps\in\R_{>0}\), we consider the following nonlinear conservation law:
\begin{align}
    \partial_{t}q(t,x)+ \partial_{x}\big(V_{\eps}(o-q)q\big)=0.\label{eq:conservation_law_smooth}
\end{align}
Here, the velocity becomes small when \(q\) approaches \(o\), and whenever \(q\) is not close to the obstacle, the dynamics evolve with the constant speed of (almost) \(1\) as expected for the conservation law without the obstacle.
However, the smoothness of \(V_{\eps}\) remains problematic, as we indeed want the dynamics to hit the obstacle and only slow down when hitting it. Thus, the singular limit \(\eps\rightarrow 0\) of the solution and the velocity \(V_{\eps}\) is of interest as well as what values \(V_{\eps}\) attains at this limit point, resulting in a discontinuous (in the solution) conservation law (compare, in particular, \cref{rem:discontinuous_conservation_law}),
\begin{equation}
  \partial_{t}q(t,x)+ \partial_{x}\big(H(o-q)q\big)=0,\label{eq:conservation_law_disc}
\end{equation}
where \(H\) denotes the Heaviside function. All the raised points are addressed in this manuscript, starting in \cref{sec:preliminaries} with the assumptions on the conservation law, the obstacle, and more (see \cref{as1}).
Additionally, as we deal with the viscous approximation of \cref{eq:conservation_law_smooth}, we look into its well-posedness, i.e.,\ existence and uniqueness of solutions.

In \cref{sec:comparison_principles}, we then show that the solutions to the viscosity approximation and \cref{eq:conservation_law_smooth} satisfy the obstacle constraint strictly as long as the initial datum is ``compatible'' and the density remains non-negative in the viscosity limit (for non-negative initial datum), uniformly in \(\eps\in\R_{>0}\).

As we are interested in the mentioned convergence for \(\eps\rightarrow 0\),\ cref{sec:Lipschitz} presents one-sided Lipschitz (OSL) bounds for the solution and the velocity, uniformly in \(\eps\in\R_{>0}\), from which we can conclude certain compactness properties of the solution and the velocity, culminating in the convergence result in \cref{alllim}. In addition, we show that even in the limit \(\eps\rightarrow 0\),\ the tuple \((q_{\eps},V_{\eps}(o-q_{\eps}))\) satisfies a (discontinuous) conservation law, which is indeed of the form stated in \cref{eq:conservation_law_disc}.

\Cref{sec:characterization_limit} then characterizes the limit solution and velocity in specific, more restrictive setups, illustrating the reasonableness of the approach and providing insights into the limiting velocity. Surprisingly, the velocity does not become zero when the density approaches the obstacle. This phenomenon may seem counter-intuitive at first glance, but assuming that the velocity is zero, no mass is transported when the obstacle is active, i.e.,  a ``full blocking'' results once the obstacle is active. 
We also showcase how a backward shock front originating from the point of contact of the obstacle and the solution propagate.

Next, \Cref{sec:motivation_optimization} continues to motivate the chosen approach for the obstacle problem, this time via optimization. On a formal level, we want to maximize the velocity of the conservation law at each point in time so the obstacle is not violated when forward propagating in time.  This leads to an optimization problem for all time \(t\in[0,T]\). We again see how the velocity behaves if the solution touches the obstacle, coinciding with the result in \cref{sec:characterization_limit}.

In \cref{sec:numerics}, we conclude the contribution with numerical approximations by means of a tailored Godunov scheme, demonstrating the reasonableness of the approach and showing the approximate solution for some specific cases for a better understanding and intuition.

\section{Preliminaries}\label{sec:preliminaries}
First, we define a sequence of functions that approximates the Heaviside function smoothly and later plays an important role in approximating a solution to the obstacle problem. These functions should have the following properties:
\begin{assumption}[Approximation of Heaviside function]\label{asH}
    To approximate the Heaviside function, we consider a sequence $V_\eps(x) := V\left(\tfrac{x}{\eps} \right)$ for all $x \in [0, \infty)$, $\eps > 0$. Here, $V \in C^\infty\big([0, \infty)\big)$ fulfills the following:
    \begin{itemize}
        \item $\lim_{x \rightarrow \infty} V(x) = 1$;
        \item $V(0) = 0$;
        \item $V\vert_{(0, \infty)} > 0$, $(-1)^{k+1}V^{(k)}\vert_{[0, \infty)} > 0$ for all $n \in \mathbb{N}$;
        \item $V^{(k)} \in L^\infty\big([0, \infty)\big)$ for all $k \in \mathbb{N}$, $\lVert V \rVert_{L^\infty\left([0, \infty)\right)} = 1$ and $\sup_{\eps > 0}\left \lVert V_\eps'\right \rVert_{L^\infty([\tilde{c}, \infty))} < \infty$ and all $\tilde{c}>0$;
        \item it holds that
        \[
        -\infty <v_{2, 1}^- \coloneqq \inf_{y \in \mathbb{R}_{\geq 0}} \tfrac{V''(y)}{V'(y)} \leq v_{2, 1}^+ \coloneqq \sup_{y \in \mathbb{R}_{\geq 0}} \tfrac{V''(y)}{V'(y)} < 0,
        \]
        and
        \item $\lim_{x \rightarrow \pm \infty} xV_\eps^{(k)}(x) = 0$ for all $k \in \mathbb{N}$.
    \end{itemize}
    From the first point, we see that $\lim_{\eps \searrow 0} V_\eps = \chi_{(0, \infty)}$ pointwise. The last point is not explicitly referenced again in this work. However, the assumption is not very restrictive and is used in the proof of~\cref{E}.
\end{assumption}
To demonstrate that the above requirements are reasonable, the following is a simple example of such function $V$.
\begin{example}[Example of approximation]\label{ex:V}
    We set $V:\mathbb{R} \rightarrow \mathbb{R}$, $V(x) := 1-\exp(-x)$ for $x \in \mathbb{R}$, where $V$ satisfies all the properties outlined in~\cref{asH}.
\end{example}
Now, as stated in the introduction, we want to investigate \cref{eq:conservation_law_smooth} and its limiting behavior for \(\eps\rightarrow0\) and \(T\in\R_{>0}\)
    with initial condition $q(0, \cdot) \equiv q_0$. To proceed, assumptions about the regularity of the data are necessary.
\begin{assumption}[Regularity]\label{as1}
    When working with~\cref{eq:conservation_law_smooth}, we assume the following:
    \begin{description}
        \item[Velocity field:] The velocity field \(V_{\eps}\) adheres to~\cref{asH}.
        \item[Initial datum:] \(q_0 \in BV(\mathbb{R};\R_{\geq 0})\)
        \item[Obstacle:] $o \in W^{6, \infty}(\mathbb{R})$ such that 
        \begin{description}
           \item[\hspace*{2cm} $\bullet$] \(o'\in W^{6,1}(\mathbb{R})\),
            \item[\hspace*{2cm} $\bullet$] $\underset{x \in \mathbb{R}}{\mathrm{essinf}}~ o(x) - q_0(x) > 0$, and
            \item[\hspace*{2cm} $\bullet$] $\lim_{x \rightarrow \pm \infty} o(x)$ both exist in $\mathbb{R}$.
        \end{description}
    \end{description}
\end{assumption}
Here $\mathrm{BV} \subseteq L^1$ denotes the space of functions with bounded variation and $\lvert \cdot \rvert_{\mathrm{TV}}$ will denote the total variation (semi-norm).\\
Existence and uniqueness to~\cref{eq:conservation_law_smooth} for \(\eps\in\R_{>0}\) can be found in the standard literature. We present the notion of solution (i.e.,\ weak entropy solution) and a precise statement about existence and uniqueness.

\begin{definition}[Entropy solution]\label{entropy}
    Let $q_0$ be as in~\cref{as1}. We say that $q \in C\big([0, T]; L^1_{loc}(\mathbb{R})\big)$ is an entropy solution to the Cauchy problem \cref{eq:conservation_law_smooth} with \(q(0, \cdot) \equiv q_0\) 
    if $q(0, \cdot) = q_0$ in $L^1_{loc}(\mathbb{R})$, and for any \((\kappa,\varphi)\in \mathbb{R}\times C_0^\infty\big((0, T) \times \mathbb{R}; \mathbb{R}_{\geq 0} \big)\), the following holds:
        \begin{align*}
        & \int^T_0 \!\!\!\int_{\mathbb{R}} | q(t, x) - \kappa | \partial_t \varphi(t, x) + \mathrm{sgn}(q(t, x) - \kappa)\big( f(x, q(t, x)) - f(x, \kappa)\big)\partial_x \varphi(t, x)   \dd x \dd t  \\
        & \quad - \int^T_0 \!\!\!\int_{\mathbb{R}} \mathrm{sgn}(q(t, x) - \kappa)\partial_1 f(x, \kappa)\varphi(t, x) \dd x \dd t \geq 0.
        \end{align*}
        where \(f(x,q) \coloneqq V_\eps(o(x)-q)q \ \forall (x,q) \in \left \lbrace (y, z) \in \mathbb{R}^2: o(y)\geq z\right \rbrace \).
\end{definition}
Given the definition of an entropy solution, we can state the existence and uniqueness of \cref{eq:conservation_law_smooth} as follows:    
    \begin{theorem}[Existence and uniqueness to~\cref{eq:conservation_law_smooth}]\label{theo:uniqueness_entropy_solution}
        There is a unique entropy solution $q_\eps \in C\big([0, T]; L^1_{loc}(\mathbb{R}) \big)$ to \cref{eq:conservation_law_smooth}  as defined in \cref{entropy}. Moreover, $0 \leq q_\eps \leq o$ a.e.\ in \((0,T)\times\R\).
    \end{theorem}
    \begin{proof}
        Existence follows from~\cite[§4]{Kruzkov1970}. In addition, ~\cite[p.229, Theorem 3]{Kruzkov1970} implies directly that $0 \leq q_\eps \leq o$ a.e., as $0$ and $o$ are both solutions to~\cref{eq:conservation_law_smooth}. Using these bounds, uniqueness can be obtained owing to the doubling of variables technique~\cite[§3]{Kruzkov1970}. 
    \end{proof}

     To obtain quantitative estimates with regard to the \(\eps\) parameter, we approximate \cref{eq:conservation_law_smooth} with a viscosity to work with smooth solutions, as detailed in the following
    \begin{definition}[Viscosity approximation]\label{defi:viscosity}
        Let $\nu, \eps \in \mathbb{R}_{>0}$ and let~\cref{as1} hold. Moreover, we set  $q_{0, \nu} = \varphi_\nu * q_0$, where $*$ denotes the convolution~\cite[4.13]{alt2016eng} and $\varphi_\nu$ is given by
        \[
        \varphi_\nu(x) \coloneqq \tfrac{1}{\nu}\varphi\left( \tfrac{x}{\nu}\right),\ x\in\R,
        \]
        where $\varphi$ is the standard mollifier~\cite[4.2.1, Notation (ii)]{evans}.\\
        Then, we call the solution \(q_{\nu,\eps}: (0,T)\times \R \rightarrow \R\) to the initial value problem
        \begin{align}
        \partial_t q(t, x) + \partial_x \big(V_\eps(o(x) -q(t, x)\big) &= \nu\partial_x^2 \big(q(t, x) - o(x)\big), & (t, x) \in (0, T) \times \mathbb{R}  \label{eq:viscosity_PDE}\\
         q_{\nu, \eps}(0, \cdot) &\equiv q_{0, \nu}, &\text{ on } \R\label{vic}
    \end{align}
    the \emph{viscosity approximation} of~\cref{eq:conservation_law_smooth}.    
    \end{definition}
    The viscosity approximation approach has been laid out by, e.g.,\ Kružkov~\cite{Kruzkov1970}. However, it is not canonical to have $\nu o''$ on the right-hand side of a viscosity approximation. In fact, this term ensures the obstacle condition $q_{\nu, \eps} < o$ even for the viscosity approximation.

    In the following, we state properties of the viscosity approximation and start with the well-posedness of \crefrange{eq:viscosity_PDE}{vic}. For existence, we resort to parabolic theory.

    \begin{theorem}[Existence for the viscous equation~\cref{eq:viscosity_PDE}]\label{E}
        If $q_0$ and $o$ satisfy assumption~\cref{as1}, then there is a unique solution $q_{\nu, \eps} \in H^5\big((0, T) \times \mathbb{R}\big)$ to the problem~\cref{eq:viscosity_PDE}-\eqref{vic}.
    \end{theorem}
    \begin{proof}
        To prove this statement, we follow the steps in~\cite[Chapter II, Section 2]{godlewski1991} while extending the arguments to a space-dependent flux (as is our case).\\
        To apply the results in~\cite{godlewski1991}, we first extend $V_\eps$ to a function $\tilde{V}_\eps \in C^\infty(\mathbb{R})$ such that $\lVert \tilde{V}_\eps \rVert_{L^\infty(\mathbb{R})} = 1$ and $\tilde{V}_\eps\vert_{(-\infty, -\ell]} = 0$ for some $\ell > 0$ and for all $0< \eps < 1$. \\
        The result is $\tilde{V}_\eps$, and after obtaining the maximum principle~\cref{mprin}, which is proved for first the extension $\tilde{V}_\eps$ as well, we obtain the result for~\cref{eq:viscosity_PDE}. This is because $\tilde{V}_\eps$ will then only be evaluated on $[0, \infty)$.
    \end{proof}
    Thanks to classical embedding theorems, we also obtain the following
    \begin{remark}[Asymptotic behavior of $q_{\nu, \eps}$]\label{asyb}
        For $\nu,\eps \in \R_{>0}$, the solution $q_{\nu, \eps}$ to problem~\cref{eq:viscosity_PDE} belonging to $H^5\big((0, T) \times \mathbb{R}\big)$ implies that 
        \(
        \lim_{x \rightarrow \pm \infty} \partial_\alpha g(t, x) = 0
        \)
        for every $t \in [0, T]$ and every multi-index $\alpha$ up to at least order $4$. This can be found, e.g.,\ in Brézis~\cite[Corollary 8.9]{brezis2010functional}. We use this fact extensively in~\cref{sec:comparison_principles}.
    \end{remark}

Now, by embedding theorems, we have a \emph{classical} solution $q_{\nu, \eps} \in C^3\big([0, T] \times \mathbb{R}\big)$ of the viscous approximation. Thus, we can obtain bounds in $L^1$ as well as $\mathrm{TV}$:
\begin{theorem}[Estimates for $q_{\nu, \eps}$]\label{tvinx}
    For \(\varepsilon,\nu\in\R_{>0}\), the solution $q_{\nu, \eps} \in H^5\big((0, T) \times \mathbb{R}\big)$ to~\cref{eq:viscosity_PDE} satisfies
    \begin{description}
        
    \item[$L^1$ bound:]  \(    \| q_{\nu, \eps} \|_{L^\infty((0, T); L^1(\mathbb{R}))} \leq \big(T\nu \| o'' \|_{L^1(\mathbb{R})} + \| q_0 \|_{L^1(\mathbb{R})}\big) 
\)
\item[Spatial \(\mathrm{TV}\)-bound:] For all \(t\in[0,T]\), the following holds:
    \begin{align*}
    \lvert q_{\nu, \eps}(t, \cdot) - o \rvert_{\mathrm{TV}(\mathbb{R})} & \leq \left( \| o' \|_{L^1(\mathbb{R})} + \lvert q_0 \rvert_{\mathrm{TV}(\mathbb{R})} \right)T \| o'' \|_{L^1(\mathbb{R})} e^{T \| V_\eps' \|_{L^\infty\left([0, \infty)] \right)}\| o' \|_{L^\infty(\mathbb{R})} }.
    \end{align*}
\item[\(\mathrm{TV}\)-bound in time:] There exists $\bar \nu >0$ such that for all \(t\in[0,T]\), 
    \begin{align*}
    \int_{\R} \lvert \partial_t q_{\nu, \eps}(t, x) \rvert \dd x  & \leq \sup_{\eps > 0}\lVert V_\eps' \rVert_{L^\infty\left(\left[\hat{K}, \infty\right)\right)} \| q_0 \|_{L^\infty(\mathbb{R})}\big(\| o' \|_{L^1(\mathbb{R})}+ \lvert q_0 \rvert_{\mathrm{TV}(\mathbb{R})} \big) \\
    & \quad + \left( 1 + 2e^{-1}\right)\lvert q_0 \rvert_{\mathrm{TV}(\mathbb{R})} + \nu \| o'' \|_{L^1(\mathbb{R})}
    \end{align*}
    for all $0 < \nu < \bar \nu$, where \(\hat{K} \coloneqq \tfrac{\mathrm{essinf}_{y \in \mathbb{R}} o(y) -q_0(y)}{2}\).
    
\end{description}

\end{theorem}
Even more, for monotonically decreasing \(o\), we even obtain uniform \(\mathrm{TV}\) bounds in \(\eps\in\R_{>0}\):
\begin{remark}[Uniform \(\mathrm{TV}\) bounds for monotone obstacles]
    Let \(o\) be decreasing and \(q_{\nu,\eps}\) be the solution as in \cref{eq:viscosity_PDE}.
    Then, in $\eps, \nu \in \mathbb{R}_{>0}$ and \(t \in (0, T)\) the following holds uniformly:
    \[
    \int_{\mathbb{R}} \lvert \partial_x q_{\nu, \eps}(t, x) -o'(x) \rvert ~\mathrm{d}x \leq \left( \| o' \|_{L^1(\mathbb{R})} + \lvert q_0 \rvert_{\mathrm{TV}(\mathbb{R})} \right)T \| o'' \|_{L^1(\mathbb{R})}.
    \]
    The case of increasing obstacles is trivial as the obstacle never becomes active.
\end{remark}

Now, we have collected the necessary tools to prove that $\lim_{\nu \rightarrow} q_{\nu,\varepsilon}$ is indeed an entropy solution to \cref{eq:conservation_law_smooth}:

\begin{theorem}[$q_{\nu,\eps}$ converges toward the entropy solution for \(\nu\rightarrow 0\)]\label{entropynu}
    Let $\eps,\nu \in\R_{>0}$ and \(q_{\eps,\nu}\) the viscosity solution in \cref{defi:viscosity}. Then, there exists \(q_{\eps}\in C\big([0,T];L^{1}_{\text{loc}}(\R)\big)\) so that
    \[
    \lim_{\nu\rightarrow 0}\|q_{\eps,\nu}-q_{\eps}\|_{C([0,T];L^{1}_{\text{loc}}(\R))}=0
    \]
    and \(q_\eps\) is the unique entropy solution to~\cref{eq:conservation_law_smooth} in the sense of~\cref{entropy}.  
\end{theorem}
\begin{proof}
    The convergence of $q_{\nu,\eps}$ for $\nu \rightarrow 0$ on subsequences follows from~\cref{tvinx} and the Riesz--Kolmogorov-type Theorem for Bochner spaces~\cite[Theorem 3]{Simon1986CompactSI}. Showing that $q_\eps$ is an entropy solution and by the uniqueness of entropy solution in \cref{theo:uniqueness_entropy_solution}, we indeed obtain convergence on all subsequences. This a standard argument that can be found in~\cite[p.236]{Kruzkov1970}. For this, 
    $\| q_{\nu, \eps} \|_{L^\infty([0, T]\times\mathbb{R})}$ must be bounded uniformly in $0 < \nu < 1$. This is proven in~\cref{sec:comparison_principles}.
\end{proof}

\section{Comparison principles}\label{sec:comparison_principles}
~\cref{theo:uniqueness_entropy_solution} shows that $q_\eps$ respects the obstacle. However, we can even prove that the obstacle is never hit. This prevents the ``full blocking'', as mentioned in the introduction.
\begin{lemma}[Maximum principle/Obstacle is respected]\label{mprin}
    Let $q_{\nu, \eps}$ be a solution to~\cref{eq:viscosity_PDE} for $\nu, \eps >0$. Then,
    \begin{equation}\label{upb}
    q_{\nu,\eps}(t, x) < o(x),
    \qquad \forall (t, x) \in [0, T] \times \mathbb{R}.
    \end{equation}
\end{lemma}
\begin{proof}
    First, extend $V_\eps$ to $\tilde{V}_\eps \in C^\infty(\mathbb{R})$ as discussed in the proof of~\cref{E}. Let $q_{\nu, \eps}$ be the solution to
    \begin{align}
        \partial_t q(t, x) + \partial_x \big(\tilde{V}_\eps(o(x) -q(t, x)\big) &= \nu\partial_x^2 \big(q(t, x) - o(x)\big), & (t, x) \in (0, T) \times \mathbb{R}  \label{eq:ext1}\\
         q_{\nu, \eps}(0, \cdot) &\equiv q_{0, \nu}, &\text{ on } \R\label{eq:ext2}.
    \end{align}
    Not only will we obtain the maximum principle for $q_{\nu, \eps}$, but the maximum principle also implies that $q_{\nu, \eps}$ solves~\cref{eq:viscosity_PDE,vic}:\\
    Fix $t \in (0, T]$. In the end, we want to use \cite[Theorem 1]{Milgrom2002}. Here, $(0, T) \mapsto \sup_{y \in \mathbb{R}} q_{\nu, \eps}(t, y) - o(y)$ must be differentiable a.e. Fortunately, this is guaranteed by Lipschitz continuity of all involved functions. Moreover, $\lim_{y \rightarrow \pm \infty} q_{\nu, \eps}(t, y) -o(y)$ has to exist, and $\lim_{y \rightarrow \pm \infty} \partial_t\big(q_{\nu, \eps}(t, y) -o(y)\big) = 0$ has to hold, both of which can be seen by invoking~\cref{as1} and~\cref{asyb}.\\
    In the spirit of~\cite[Theorem 1]{Milgrom2002}, set
    \[
    X(t) \coloneqq \left \lbrace x \in \mathbb{R} \cup \lbrace -\infty, \infty \rbrace: q_{\nu, \eps}(t, x) - o(x) = \sup_{y \in \mathbb{R}} q_{\nu, \eps}(t, y) - o(x) \right \rbrace \neq \emptyset
    \]
    for $t \in [0, T]$. Set $m(t) \coloneqq \sup_{y \in \mathbb{R}} q_{\nu, \eps}(t, y) - o(y)$ for $t \in [0, T]$ and choose any $x \in X(t)$.\\
    Then, \textbf{as a first case}, if $x \in \lbrace - \infty, \infty\rbrace$, we have
    \begin{align}
    \lim_{y \rightarrow \infty} \partial_t \big( q_{\nu, \eps}(t, y) -o(y)\big) = \lim_{y \rightarrow \infty} \partial_t q_{\nu, \eps}(t, y) = 0 = \partial_t \big( q_{\nu, \eps}(t, x) - o(x)\big). \label{eq:dt_q-o_infty}
    \end{align}
    In the \textbf{second case}, if $x \in \mathbb{R}$, then, according to standard optimality conditions~\cite[Chapter 2]{nocedal2006numerical}, we obtain
    \begin{equation}\partial_x q_{\nu, \eps}(t, x) - o'(x) = 0\qquad \wedge\qquad\partial_x^2 q_{\nu, \eps}(t, x) - o''(x) \leq 0\label{eq:optimality_obstacle}\end{equation}
    Then, we can differentiate as follows:
    \begin{align*}
    \partial_t \big(q_{\nu, \eps}(t, x) -o(x) \big) &= \partial_t q_{\nu, \eps}(t, x) \\
    &\stackrel{\eqref{eq:conservation_law_smooth}}{=}-\tilde{V}_\eps'(o(x) -q_{\nu, \eps}(t, x)) \underbrace{(o'(x) -\partial_x q_{\nu, \eps}(t, x))}_{=0} q_{\nu, \eps}(t, x) \\
    &\qquad   - \tilde{V}_\eps(\underbrace{o(x) -q_{\nu, \eps}(t, x)}_{=-m(t)}) \underbrace{\partial_xq_{\nu, \eps}(t, x)}_{=o'(x)} + \nu \underbrace{(\partial_x^2 q_{\nu, \eps}(t, x) -o''(x))}_{\leq 0} \\
    \intertext{Using \cref{eq:optimality_obstacle},}
    & \leq \left\lvert \tilde{V}_\eps\big(-m(t)\big) \right \rvert \| o' \|_{L^\infty(\mathbb{R})}
    \end{align*}
    Now,  $q_{\nu, \eps}-o \in W^{2, \infty}\big([0, T] \times \mathbb{R}\big)$, and
    \(
    \lim_{x \rightarrow \pm \infty} q_{\nu, \eps}(t, x) - o(x)
    \)
    exists for all $t \in [0, T]$ because of~\cref{as1} and~\cref{asyb}. Similarly, 
    \[
    \lim_{x \rightarrow \pm \infty} \partial_t\big( q_{\nu, \eps}(t,x) - o(x) \big) = 0\ \forall t\in[0,T].
    \]
    Altogether, \cite[Theorem 1]{Milgrom2002} is applicable, and combining both cases for a.e.\ $t \in (0, T)$ and all $x\in X(t)$ yields
    \begin{align*}
        \tfrac{\mathrm{d}}{\mathrm{d}t} \sup_{y \in \mathbb{R}} q_{\nu, \eps}(t, y) - o(y)  = m'(t)  \leq  \partial_t\big(q_{\nu, \eps}(t, x) - o(x)\big) = \| o' \|_{L^\infty(\mathbb{R})} \left 
        \lvert \tilde{V}_\eps\big( -m(t)\big) \right \rvert
    \end{align*}
    Now, consider the Cauchy problem for the corresponding differential equation (instead of the inequality)
    \begin{equation*}
    \tilde{m}'(t) = \| o' \|_{L^\infty(\mathbb{R})}\left \lvert \tilde{V}_\eps(-\tilde m(t))\right \rvert, \quad \tilde m(0) = \sup_{y\in \mathbb{R}} q_{0, \nu}(y) - o(y) 
    \end{equation*}
    on $[0, T]$ that uniquely defines $\tilde m$ on $[0, T]$ because $\tilde{V}_\eps$ is globally Lipschitz (see~\cref{asH}). Since $\tilde{V}_\eps(0) = 0$ and $-\tilde m(0) > 0$,  $-\tilde m(t) > 0 $ for all $t \in [0, T]$ because equilibria are not allowed to be crossed. We also note that \emph{$\tilde{m}$ does not depend on $\nu$!} Then, the ODE comparison principle (see, e.g., \cite[Lemma 1.2, p.24]{teschlbook}) implies for $(t, y) \in [0, T] \times \mathbb{R}$,
    \[
    q_{\nu, \eps}(t, y) - o(y) \leq m(t) \leq \tilde m(t) < 0
    \]
    This yields the claim.
\end{proof}
So, in fact, the obstacle is not even hit at all for \(\eps\in\R_{>0}\), which also implies that the mass is transported with a positive velocity.

The same argument also allows for the proof of a minimum principle:
\begin{theorem}[Minimum principle]\label{minpr}
    Let $q_{\nu, \eps}$ be a solution to~\cref{eq:viscosity_PDE} for $\nu, \eps >0$. Then,
    \begin{equation}\label{lowb}
    q_{\nu, \eps}(t, x) \geq -\tfrac{\nu\| o'' \|_{L^\infty(\mathbb{R})} }{\| o' \|_{L^\infty(\mathbb{R})} \tilde{K}}\left(-1+ \exp\left(\tilde{K}\| o' \|_{L^\infty(\mathbb{R})} t\right)\right)
    \end{equation}
    for all $(t, x) \in [0, T] \times \mathbb{R}$, where
    \(
    c_0 \coloneqq \inf_{x \in \mathbb{R}} o(x) > 0,~\tilde{K} := \lVert V_\eps' \rVert_{L^\infty\left(\left[c_0, \infty \right) \right)}.
    \)
    We recall~\cref{asH} to know that $\tilde{K}$ is finite.   
\end{theorem}
\begin{proof}
The proof is similar to the previous one that shows that the obstacle is respected.
It is again based on \cite[Theorem 1]{Milgrom2002}.
\end{proof}
Of note, the solution can be negative for \(\nu,\eps\in\R_{>0},\); however, in the limit, it will satisfy the expected lower bound of \(0\).

This also allows us to state bounds for the entropy solution:
\begin{remark}[Bounds on entropy solution]
    Let $\eps > 0$. Because $L^1_{loc}$-convergence implies point-wise convergence a.e.\ of a subsequence, the entropy solution $q_\eps\in L^1_{loc}(\mathbb{R})$ from~\cref{entropynu} fulfills the bounds
    \[
    0 \leq q_\eps(t, x) < o(x)\quad   \text{for a.e. } (t, x) \in (0, T) \times \mathbb{R}.
    \]
This can be obtained by letting $\nu \searrow 0$ in~\cref{lowb} and~\cref{upb} point-wise a.e.. Here, we note that an upper bound $\tilde{m}$ exists on $q_{\nu, \eps}-o$ such that $\tilde{m}<0$ and does not depend on $\nu$, as stated in the proof of~\cref{mprin}. This allows for the strict inequality $q_\eps < o$ to hold pointwise a.e.
\end{remark}
 
\section{OSL bounds for \texorpdfstring{$q_{\nu, \eps}$}{q}\label{s4}
 and \texorpdfstring{$V_\eps(o - q_{\nu, \eps})$}{V(o-q)} uniformly  \texorpdfstring{in \(\nu,\eps\)}{}}\label{sec:Lipschitz}

As stated earlier,  the condition $o' \leq 0$ gives uniform $\mathrm{TV}$-bounds of $q_{\nu, \eps}$ in $\nu$ and $\eps$. We now want to find OSL bounds that provide a more general compactness result in $\nu$, $\eps > 0$. We need the following:
\begin{lemma}[OSL bounds and compactness]\label{tvo}
    Let $\Omega\subseteq \mathbb{R}$ be open and $(f_\mu)_{\mu>0} \subseteq C^1( \Omega )$ be a sequence of functions such that
    \[
    C_1 \coloneqq \sup_{\mu > 0} \| f_\mu \|_{L^\infty(\Omega)} < \infty \text{  and  } C_2 \coloneqq \sup_{\mu > 0} \sup_{x \in \Omega}  f_\mu'(x) < \infty.
    \]
    Then, $(f_\mu)_{\mu >0} \subseteq BV(V)$ for every compact $V \subseteq \Omega$, and there exist $f \in L^1_{\text{loc}}(\Omega)$ and a subsequence \((\mu_{k})_{k\in\N}\text{ fulfilling } \lim_{k\rightarrow\infty} \mu_{k}=0\) such that
    \begin{align*}
        \lim_{k\rightarrow \infty} \| f_{\mu_k}- f \|_{L^1_{\text{loc}}(\Omega)} = 0.
    \end{align*}
\end{lemma}
\begin{proof}
    First, we prove that $| f_\mu |_{\mathrm{TV}(V)}$ is finite and does not depend on $\mu$ for every compact $V \subseteq \overline{\Omega}$.
    Let $V \subseteq \Omega$ be compact, and a compact interval $K \supseteq V$. Set $f_\mu^+ \coloneqq \max \lbrace f_\mu, 0 \rbrace$ and $f_\mu^- \coloneqq - \lbrace f_\mu, 0 \rbrace$ for any $\mu > 0$.
    \begin{align*}
        \int_V | f_\mu'(x) |~\mathrm{d}x & \leq \int_K | f_\mu'(x)| ~\mathrm{d}x = \int_K  f_\mu'( x)^+ + f_\mu'( x)^-~\mathrm{d}x  \\
        & = 2\int_K  f_\mu'( x)^+ ~\mathrm{d}x - \int_K  f_\mu( x)'~\mathrm{d}x \\
        & = 2 \int_K \underbrace{ f_\mu'( x)^+}_{\leq C_2}~\mathrm{d}x - f_\mu(\sup K) + f_\mu(\inf K)  \leq 2 | K | C_2 + 2 C_1 
    \end{align*}
    Consequently, \(
    | f_\mu |_{\mathrm{TV}(V)} \leq 2| K | C_2 + 2C_1\),
    and per \cite[Theorem 13.35]{leoni2009first}, the classical compactness result for functions of bounded variation gives convergence to $f$ in $L^1_{\text{loc}}(\Omega)$ along a subsequence. 
\end{proof}
Note that the result from~\cref{tvo} also holds if $\inf_{\mu> 0} \inf_{x \in \Omega} f'_\mu(x) >- \infty$ instead of $\sup_{\mu > 0} \sup_{x \in \Omega}  f_\mu'(x) < \infty$. To obtain this, replace $f$ with $-f$ and use~\cref{tvo}.\\
Unfortunately, linear conservation laws have no intrinsic ``smoothing'' because of the lack of Oleinik-type estimates. However, as we later want to pass to the limit nevertheless, we impose an OSL condition on the initial datum, which is preserved over time (the same cannot easily be obtained with a \(\mathrm{TV}\) argument).    
\begin{assumption}[Initial datum revisited]\label{as2}
In addition to \cref{as1}, we assume for the remainder of~\cref{s4} that the initial datum is OSL continuous from below, i.e., for a \(C\in\R\), the following holds:
\begin{equation}\label{oslq0}
\inf_{(x,y)\in\R^{2}}\tfrac{q_{0}(x)-q_{0}(y)}{x-y}\geq C.
\end{equation}
\end{assumption}
In the following, we derive OSL bounds for $q_{\nu,\eps}$ uniform in $\nu,\eps$:
\begin{theorem}[OSL condition for $q_{\nu, \eps}$ in  space] \label{oslq}
    There exists $\nu(\eps) \in \R_{>0}$ with $\lim_{\eps \searrow 0}\nu(\eps) = 0$ such that
    $\forall \nu \in \R_{>0}$  with $0 \leq \nu \leq \nu(\eps)$, and it holds \(\forall (t, x) \in [0, T] \times \mathbb{R}\) that
\[
    \partial_x q_{\nu, \eps}(t, x) -o'(x) > \min \big\lbrace -2 \| o' \|_{L^\infty(\mathbb{R})}, \essinf_{y\in \R}\big(q_0'(y) - o'(y)\big) \big\rbrace - T \| o'' \|_{L^\infty(\mathbb{R})},
    \]
      where the right-hand side term is bounded thanks to \cref{as2}.
      
    Consequently, with $o' \in L^\infty(\mathbb{R})$---see \cref{as1}---$\forall \nu \in \R_{>0}$ with  $0 \leq \nu \leq \nu(\eps)$ and  $\forall (t, x) \in [0, T] \times \mathbb{R}$,
    \[
    \partial_x q_{\nu, \eps}(t, x) > \min \lbrace -2 \| o' \|_{L^\infty(\mathbb{R})}, \essinf_{y\in \R}\big(q_0'(y) - o'(y)\big) \rbrace - T \| o'' \|_{L^\infty(\mathbb{R})} - \| o' \|_{L^\infty(\mathbb{R})}.
    \]
\end{theorem}
Here, $\essinf_{y\in \R}\big(q_0'(y) - o'(y)\big)$ is interpreted as the lower OSL bound of $q_0 - o$. 
\begin{proof}
    As before, we apply~\cite[Theorem 1]{Milgrom2002}. Let
    \[
    Z(t) \coloneqq \left \lbrace x \in \mathbb{R} \cup \lbrace -\infty, \infty \rbrace: \partial_x q_{\nu, \eps}(t, x) - o'(x) = \inf_{y \in \mathbb{R}} \big(\partial_y q_{\nu, \eps}(t, y) - o'(y) \big)\right \rbrace
    \]
    and $m(t) \coloneqq \big(\inf_{y \in \mathbb{R}} \partial_y q_{\nu, \eps}(t, y) - o'(y)\big)$ for all $t  \in [0,T]$. Choose $x \in Z(t)$. We distinguish two cases:\\
    The \textbf{first case} is $x \in \lbrace -\infty, \infty \rbrace \cap Z(t)$. Here, 
    \[
    \partial_t \big( \partial_x q_{\nu, \eps}(t, x) - o'(x) \big) = 0
    \]
    since \cref{eq:dt_q-o_infty} and $o' \in H^1(\mathbb{R})$ as in \cref{as1}. \\
    The \textbf{second case} is $x \in \mathbb{R} \cap Z(t)$. Once again, optimality conditions lead to
    \begin{equation}
    \partial_x^2 q_{\nu, \eps}(t,x) - o''(x) = 0\ \text{ and }\ \partial_x^3 q_{\nu, \eps}(t, x) - o'''(x) \geq 0\label{eq:optimality_42},
    \end{equation}
     and we then obtain
    \begin{align*}
        &\partial_t \big( \partial_x q_{\nu, \eps}(t, x) - o'(x)\big)\notag\\
        & = \partial_{xt}^{2} q_{\nu, \eps}(t, x) \notag. \\
        \intertext{Plugging in the strong form \cref{eq:viscosity_PDE}, we have }
        & = \partial_x \big( -V_\eps'(o(x) -q_{\nu,\eps}(t, x)) (o'(x) -\partial_xq_{\nu, \eps}(t, x))q_{\nu, \eps}(t, x) \notag \\
        & \quad - V_\eps(o(x) -q_{\nu, \eps}(t, x)) \partial_xq_{\nu, \eps}(t, x) + \nu(\partial_x^2 q_{\nu, \eps}(t, x) -o''(x))\big) \notag \\
        & = -V_\eps''(o(x) -q_{\nu, \eps}(t, x))\underbrace{(o'(x) -\partial_xq_{\nu, \eps}(t, x))^2}_{=m(t)^2} q_{\nu, \eps}(t, x) \notag \\
        & \quad - V_\eps'(o(x) -q_{\nu, \eps}(t, x)) \underbrace{(o''(x) -\partial_x^2 q_{\nu, \eps}(t, x))}_{=0} q_{\nu, \eps}(t, x) \notag\\
        & \quad - 2V_\eps'(o(x) -q_{\nu, \eps}(t, x)) \underbrace{(o'(x) -\partial_x q_{\nu, \eps}(t, x))}_{=-m(t)}\underbrace{\partial_x q_{\nu, \eps}(t, x)}_{m(t)+o'(x)} \notag \\
        & \quad  -V_\eps(o(x) -q_{\nu, \eps}(t, x)) \underbrace{\partial_x^2 q_{\nu, \eps}(t, x)}_{=o''(x)} + \underbrace{\nu(\partial_x^3 q_{\nu, \eps}(t, x) -o'''(x) )}_{\geq 0} \notag \\
        & \stackrel{\eqref{eq:optimality_42}}{\geq} -V_\eps''(o(x) -q_{\nu, \eps}(t, x))m(t)^2 q_{\nu, \eps}(t, x)  + V_\eps'\big(o(x) - q_{\nu, \eps}(t, x)\big)\big( 2m(t)^2 + m(t) o'(x) \big) \notag \\
        & \quad - V_\eps(o(x) -q_{\nu, \eps}(t, x)) o''(x).  \notag
        \intertext{We write $q_{\nu, \eps} = q_{\nu, \eps}^+ - q_{\nu, \eps}^-$ for the positive and negative parts, respectively, and incorporate $-V_\eps''m^2 q_{\nu, \eps}^+ \geq 0$:}
        &  V_\eps''(o(x) -q_{\nu, \eps}(t, x))m(t)^2 q_{\nu, \eps}(t, x)^-  + V_\eps'\big(o(x) - q_{\nu, \eps}(t, x)\big)\big( 2m(t)^2 + m(t) o'(x) \big) \notag \\
        & \quad - V_\eps(o(x) -q_{\nu, \eps}(t, x)) o''(x)
        \intertext{Per~\cref{as1}, $ V''_\eps(x) = \tfrac{1}{\eps^2}\tfrac{V''\left(\tfrac{x}{\eps}\right)}{V'\left(\tfrac{x}{\eps}\right)} V'\left(\tfrac{x}{\eps}\right) \geq \tfrac{1}{\eps}\underbrace{v_{2, 1}^-}_{<0}V_\eps'(x)$ holds for $x \in \mathbb{R}$, and we estimate as follows}
        & \geq V_\eps'\big(o(x) -q_{\nu, \eps}(t,x)\big) \left( m(t)^2\tfrac{v_{2, 1}^-q_{\nu, \eps}(t, x)}{\eps} + 2m(t)^2 +m(t) o'(x)\right) \\
        & \quad - \| V_\eps \|_{L^\infty(\mathbb{R})}\| o'' \|_{L^\infty(\mathbb{R})}. 
        \intertext{According to~\cref{lowb}, there exists $\nu(\eps)$ s.t. $q_{\nu, \eps}^- <  \tfrac{\varepsilon}{v_{2, 1}^-}$ for all $0 < \nu < \nu(\eps)$. Because $\lim_{y \rightarrow \pm \infty} \partial_yq_{\nu,\eps}(t, y)-o'(y)=0$, $m(t) \leq 0$ holds, so we can continue the estimate as}
        & \geq V_\eps'\big(o(x)-q_{\nu, \eps}(t, x)\big) \big( m(t)^2 + \| o' \|_{L^\infty(\mathbb{R})} m(t)\big) - \| o'' \|_{L^\infty(\mathbb{R})} \| V_\eps \|_{L^\infty(\mathbb{R})}.
    \end{align*}
    Applying~\cite[Theorem 1]{Milgrom2002} and combining both cases, we obtain, for a.e.\ $t \in (0, T)$,
    \begin{align*}
        m'(t) & \geq \min \!\big \lbrace V_\eps'\big(o(x)-q_{\nu, \eps}(t, x)\big) \big( m(t)^2 \!+ \! \| o' \|_{L^\infty(\mathbb{R})} m(t)\big) - \| o'' \|_{L^\infty(\mathbb{R})} \| V_\eps \|_{L^\infty(\mathbb{R})}, 0  \big \rbrace. 
    \end{align*}    Then, fix any $t \in [0, T]$. Now, if $m(t) \geq- 2 \| o' \|_{L^{\infty}(\R)}$, we have a desirable lower bound on $m$.\\
    However, if $m(t) < -2 \| o' \|_{L^\infty(\mathbb{R})}$, then $m(t)^2 + \lVert o' \rVert_{L^\infty(\mathbb{R})} m(t)\geq 0$. So,
    \begin{equation}\notag
    m'(t) \geq - \| o'' \|_{L^\infty(\mathbb{R})}\| V_\eps \|_{L^\infty(\mathbb{R})} =  - \| o'' \|_{L^\infty(\mathbb{R})}
    \end{equation}
    holds, which implies, after integration,
    \begin{align*}
    m(t) & \geq \min \lbrace m(0) , -2 \| o' \|_{L^\infty(\mathbb{R})} \rbrace - t \| o'' \|_{L^\infty(\mathbb{R})},
    \end{align*}
    which is the claimed inequality. Here, we remind the reader that $m(0) > - \infty$ is assumed in~\cref{oslq0}.
\end{proof}

A similar result can be derived for the velocity $V_\eps(o-q_{\nu, \eps})$.
 \begin{theorem}(OSL condition for $V_\eps(o-q_{\nu,\eps})$ in space)\label{oslV}
    Let \(\eps,\nu\in\R_{>0}\), $q_{\eps, \nu}$ be a solution to~\cref{eq:viscosity_PDE}. Then, for fixed $\eps$, there exists $\nu(\eps) > 0$ with $\lim_{\eps \searrow 0} \nu(\eps) = 0$ such that
    \begin{align*}
    & \partial_x \left( V_\eps\big(o(x) - q_{\nu, \eps}(t, x)\big)\right)\\
    &\leq \max \left \lbrace C_1, C_2 \sup_{\eps > 0} \left \lVert V_\eps' \right \rVert_{L^\infty\left([k_0, \infty) \right)},\esssup_{y\in \R}(o'(y) - q_0'(y)) \sup_{\eps > 0} \lVert V_\eps' \rVert_{L^\infty\left([c_0, \infty) \right) } \right \rbrace
    \end{align*}
    for all  $0 <\nu \leq \nu(\eps)$, $(t, x) \in (0, T) \times \R$, where
    \begin{align}
        k_0 & \coloneqq \tfrac{\inf_{y \in \mathbb{R}}o(y)}{2} \label{const42} \\
        c_0 & \coloneqq \essinf_{y\in \R} (o(y)-q_0(y)) \\
        v & \coloneqq \left \lVert \tfrac{V''}{V'} \right \rVert_{L^\infty(\mathbb{R})} \\
        C_1 &\coloneqq \sqrt{\tfrac{(v+2\lVert V' \rVert_{L^\infty\left([0, \infty)] \right)})^2\lVert o' \rVert^2_{L^\infty(\mathbb{R})}}{\big(v_{2, 1}^+\inf_{y \in \mathbb{R}}o(y)\big)^2}+\tfrac{2\lVert V' \rVert_{L^\infty\left( [0, \infty)\right)} (-\lVert o'' \rVert_{L^\infty(\mathbb{R})}-1)}{v_{2, 1}^+\inf_{y \in \mathbb{R}} o(y)}} \\
        & \quad +\tfrac{2(v+2\lVert V' \rVert_{L^\infty\left([0, \infty) \right)})\| o' \|_{L^\infty(\mathbb{R})}}{v_{2, 1}^+\inf_{y \in \mathbb{R}} o(y)} \label{const1} \\
        C_2 & \coloneqq T\| o''\|_{L^\infty(\mathbb{R})} - \min\big\lbrace-2\| o' \|_{L^\infty(\mathbb{R})}, \essinf_{y\in \R}\big(q_0'(y) - o'(y)\big)\big\rbrace \label{const3}
    \end{align}
\end{theorem}
\begin{proof}
As before, we want to invoke~\cite[Theorem 1]{Milgrom2002}. Fix $t \in (0, T)$ and $ \nu, \eps > 0$. Since
\[
\partial_x \big(V_\eps(o(x) -q_{\nu, \eps}(t, x))\big) = V_\eps'(o(x) -q_{\nu, \eps}(t, x))(o'(x) - \partial_t q_{\nu, \eps}(t, x)),
\]
for all $x \in \mathbb{R}$, from $\lim_{x \rightarrow \pm\infty} o'(x) -\partial_x q_{\nu, \eps}(t, x) = 0$ and the boundedness of $V_\eps'$, we obtain
\[
\lim_{x \rightarrow \pm \infty} \partial_x \big(V_\eps(o(x) -q_{\nu, \eps}(t, x))\big)   = 0.
\]
The same argument can be employed for derivatives thereof. Therefore,~\cite{Milgrom2002} is indeed applicable.

Now, choose  \(x\in X(t)\), where
\[
X(t)  \coloneqq \left \lbrace x \in \mathbb{R} \cup \lbrace - \infty, \infty \rbrace: \partial_x V_\eps \big(o(x) -q_{\nu, \eps}(t, x) \big) = \sup_{y \in \mathbb{R}} \partial_y V_\eps\big(o(y) -q_{\nu, \eps}(t, y) \big)\right \rbrace.
\]
We distinguish two cases. 
\begin{description}
    \item[Case 1:] ($x \in \lbrace - \infty, \infty \rbrace$) Then, 
    \begin{align*}
        &\lim_{y \rightarrow \infty} \partial_t \partial_y\big( V_\eps(o(y) -q_{\nu, \eps}(t, y))\big) \\
        & =\lim_{y \rightarrow \infty}  V_\eps''(o(y) -q_{\nu, \eps}(t, y))(-\partial_t q_{\nu, \eps}(t, y))(o'(y) - \partial_y q_{\nu, \eps}(t, y)) \\
        &\qquad\qquad\qquad + V_\eps'(o(y) -q_{\nu, \eps}(t, y))(-\partial_{ty} q_{\nu, \eps}(t, y)) = 0 ,
    \end{align*}
    which follows from~\cref{asyb}.
    \item[Case 2:] $(x \in \mathbb{R})$
For the sake of clarity, we (have to) introduce the \emph{following abbreviations} (for $k \in \mathbb{N}_{\geq 0}$, multi-index $\alpha$):
\begin{align*}
    V_\eps^{(k)}  & \coloneqq V_\eps^{(k)} (o(x) -q_{\nu, \eps}(t, x)), \ q_{\nu, \eps}  \coloneqq q_{\nu, \eps}(t, x), \ 
    o^{(k)}  \coloneqq o^{(k)}(x),\
    \partial_\alpha q_{\nu, \eps}  \coloneqq \partial_\alpha q_{\nu, \eps}(t, x) \\
    V^{(k)} & \coloneqq V^{(k)}\left(\frac{o(x) -q_{\nu, \eps}(t, x)}{\eps} \right) = \eps^k V_\eps^{(k)}
\end{align*}
Note that $V^{(k)}$ needs to be \emph{sharply distinguished} from $V_\eps^{(k)}$; see also~\cref{asH}. Moreover, the notation $V^{(k)}$ is a bit misleading, as it still depends on $\eps$. For $V^{(k)}$, however, we have uniform $L^\infty$-bounds in $\eps$, which further motivates omitting the index.

In the following, we sometimes write $V_\eps'\cdot (o''-\partial_x^2 q_{\nu, \eps})$ rather than, e.g.,\ $V_\eps'(o-q_{\nu,\eps})(o''-\partial_x^2 q_{\nu, \eps})$, i.e., \(V_{\eps}\) without its argument. Remember that $X(t)$ is defined so that $\forall x\in X(t)$,
\[
\partial_x V_\eps\big( q_{\nu,\eps}(t, x) - o(x) \big) = \sup_{y \in \mathbb{R}}\partial_{y}V_\eps\big( q_{\nu,\eps}(t, y) - o(y) \big).
\]
So, based on standard optimality conditions, the following must hold:
 \begin{equation}\partial_x^2 V_\eps = \partial_x^2 \big(V_\eps(o(x) -q_{\nu, \eps}(t, x))\big) = 0\quad\wedge\quad
   \partial_x^3 V_\eps = \partial_x^3 \big(V_\eps(o(x) -q_{\nu, \eps}(t, x))\big) \leq 0.
   \label{eq:optimality_conditions_for_the_n_th_time}
   \end{equation}
Expanding the above expressions (in the abbreviated notation), we obtain
\begin{align*}
    \partial_x^2  V_\eps  & = V_\eps''\cdot (o' -\partial_x q_{\nu, \eps})^2 + V_\eps' \cdot\big(o'' - \partial_x^2 q_{\nu, \eps}\big) \\
    \partial_x^3  V_\eps  & = V_\eps'''\cdot (o'-\partial_x q_{\nu, \eps})^3 + 3V_\eps''\cdot(o'-\partial_x q_{\nu, \eps})(o''-\partial_x^2 q_{\nu,\eps}) + V_\eps'\cdot (o'''-\partial_x^3 q_{\nu, \eps}),
\end{align*}
and setting $\displaystyle m(t) \coloneqq \sup_{y \in \mathbb{R}} \partial_y V_\eps\big( q_{\nu,\eps}(t, y) - o(y) \big) =  V_\eps\big( q_{\nu,\eps}(t, x) - o(x) \big)$ yields
\begin{align}
    \partial_t \partial_x V_\eps & = V_\eps''\cdot (o'-\partial_x q_{\nu, \eps}) (-\partial_t q_{\nu, \eps}) + V_\eps'\cdot (-\partial_{tx}q_{\nu, \eps}) \notag \\
    & \stackrel{\eqref{eq:conservation_law_smooth}}{=} V_\eps''\cdot (o'-\partial_x q_{\nu, \eps})\big(V_\eps'\cdot (o'-\partial_x q_{\nu, \eps}) q_{\nu, \eps}+ V_\eps\cdot \partial_x q_{\nu, \eps} - \nu (\partial_x^2 q_{\nu, \eps} -o'') \big) \notag \\
    & \quad + V_\eps' \cdot \partial_x \big(V_\eps'\cdot (o'-\partial_x q_{\nu, \eps})q_{\nu, \eps} + V_\eps \cdot \partial_x q_{\nu, \eps} -\nu(\partial_x^2 q_{\nu, \eps} -o'') \big) \notag \\
    & = V_\eps''\cdot (o'-\partial_x q_{\nu, \eps})\big(V_\eps'\cdot (o'-\partial_x q_{\nu, \eps}) q_{\nu, \eps} \notag  + V_\eps\cdot \partial_x q_{\nu, \eps} - \nu (\partial_x^2 q_{\nu, \eps} -o'') \big) \notag \\
    & \quad + V_\eps' \cdot \big( V_\eps''\cdot (o'-\partial_x q_{\nu, \eps})^2 q_{\nu, \eps} + V_\eps' \cdot (o''-\partial_x^2 q_{\nu, \eps})q_{\nu, \eps} \notag\\
    & \quad + 2V_\eps'\cdot (o'-\partial_x q_{\nu, \eps})\partial_x q_{\nu, \eps}  + V_\eps\cdot \partial_x^2 q_{\nu, \eps} - \nu( \partial_x^3 q_{\nu,\eps} - o''')\big).\label{2ord} 
    \intertext{We apply \cref{eq:optimality_conditions_for_the_n_th_time} to~\cref{2ord}, use $\partial_x q_{\nu, \eps} = -(o'-\partial_x q_{\nu, \eps})+o'$ twice,  $\partial_{xx}^2 q_{\nu, \eps} = -(o''-\partial_{xx}^2 q_{\nu, \eps})+o''$, and expand to arrive at the estimate}
    & \leq  V_\eps'' \cdot V_\eps'\cdot (o'-\partial_x q_{\nu, \eps})^2 q_{\nu, \eps} - V_\eps''\cdot V_\eps\cdot (o'-\partial_x q_{\nu, \eps})^2 + V_\eps''\cdot V_\eps\cdot  (o'-\partial_x q_{\nu, \eps}) o' \notag \\
    & \quad - \nu V_\eps''\cdot (o'-\partial_x q_{\nu, \eps})(\partial_x^2 q_{\nu, \eps} -o'')  - 2\big(V_\eps'\cdot (o'-\partial_x q_{\nu, \eps}) \big)^2 \notag \\
    &\quad + 2(V_\eps')^2\cdot  (o'-\partial_x q_{\nu, \eps})o'  -V_\eps\cdot V_\eps' \cdot (o''-\partial_x^2 q_{\nu, \eps}) + V_\eps'\cdot V_\eps\cdot  o'' \\
    &\quad - \nu V_\eps'\cdot  (\partial_x^3 q_{\nu, \eps} -o'''). \notag\\
    \intertext{Using \cref{eq:optimality_conditions_for_the_n_th_time} as well as the identities  $\partial_x^2 q_{\nu, \eps} -o''= \frac{1}{V_\eps'}\left(\partial_x^2 V_\eps -V_\eps''\cdot (o'-\partial_x q_{\nu, \eps})^2\right)$ and $ V_\eps' \cdot (o''-\partial_x^2 q_{\nu, \eps})=\partial_x^2 V_\eps - V_\eps''\cdot (o'-\partial_x q_{\nu, \eps})^2$ , we obtain}
    & = \tfrac{V_\eps'' }{V_\eps'} \big(V_\eps'\cdot (o'-\partial_x q_{\nu, \eps})\big)^2 q_{\nu, \eps} - \tfrac{V_\eps'' V_\eps}{(V_\eps')^2}\big( V_\eps'\cdot (o'-\partial_x q_{\nu, \eps})\big)^2 \notag \\
    & \quad  -\tfrac{V_\eps''}{V_\eps'}V_\eps V_\eps' \cdot (o'-\partial_x q_{\nu, \eps})o'  + \nu \tfrac{(V_\eps'')^2}{(V_\eps')^4} \big(V_\eps'\cdot  (o'-\partial_x q_{\nu,\eps})\big)^3 \notag \\
    & \quad - 2\big(V_\eps' \cdot (o'-\partial_x q_{\nu, \eps}) \big)^2 + 2(V_\eps')^2 \cdot (o'-\partial_x q_{\nu, \eps})o' \notag \\
    & \quad + \tfrac{V_\eps V_\eps'' }{(V_\eps')^2} \big(V_\eps'\cdot (o'-\partial_x q_{\nu, \eps}) \big)^2 +V_\eps' V_\eps o''\notag\\
    & \quad -\nu \tfrac{V_\eps'''}{(V_\eps')^3}(V_\eps'\cdot (o'-\partial_x q_{\nu, \eps}))^3  - 3\nu V_\eps''\cdot (o'-\partial_x q_{\nu, \eps})(o''-\partial_x^2 q_{\nu, \eps}),\notag\\
    \intertext{and eventually replacing $V_\eps'\cdot (o'-\partial_x q_{\nu, \eps})$ with $m(t)$ in most cases yields}
    & \leq \tfrac{V_\eps''}{V_\eps'} m(t)^2 q_{\nu, \eps} -\tfrac{V_\eps''}{V_\eps'}V_\eps o' m(t) + \nu \tfrac{(V_\eps'')^2}{(V_\eps')^4}\big( V_\eps'\cdot (o'-\partial_x q_{\nu, \eps})\big)^3 - 2m(t)^2 \notag \\
    & \quad + 2V_\eps'o'm(t) + V_\eps'V_\eps o''  -\nu \tfrac{V_\eps'''}{(V_\eps')^3}m(t)^3 + 3 \nu \tfrac{(V_\eps'')^2}{V_\eps'}\cdot (o'-\partial_xq_{\nu, \eps})^3. \notag
    \intertext{We factor out $V_\eps'$ for several expressions. Additionally, for $0\leq \nu \leq \nu(\eps) \leq \eps $ according to~\cref{oslq}, we have $C>0$ independent of $\nu, \eps > 0$ (see~\cref{tic}) such that $\partial_x q_{\nu, \eps} - o' \geq  C$ for all $(t, x) \in [0, T] \times \mathbb{R}$. Therefore, we can estimate $(o'-\partial_x q_{\nu, \eps})^3 \leq C^3$ if $\nu$ is sufficiently small:}
    & \leq -2m(t)^2 + V_\eps'\bigg( \tfrac{V_\eps''}{(V_\eps')^2}m(t)^2 q_{\nu, \eps}-\tfrac{V_\eps V_\eps''}{(V_\eps')^2}o'm(t) +\nu \tfrac{(V_\eps'')^2}{(V_\eps')^2}C^3 +2m(t)o' \label{eq:tbe1}  \\ 
    & \qquad\qquad \qquad \qquad \qquad  +V_\eps o''  - \nu \tfrac{V_\eps'''}{(V_\eps')^4}m(t)^3 + 3 \nu \tfrac{(V_\eps'')^2}{(V_\eps')^2}C^3 \bigg) \label{eq:tbe2}
    \end{align}
\end{description}
    Therefore, combining both cases, we show that by using \cite[Theorem 1]{Milgrom2002}, for a.e. $t \in (0, T)$,
    \begin{equation}
    \begin{split}
        m'(t) &\leq \sup_{x \in X(t)} \max \bigg \lbrace -2m(t)^2 + V_\eps'\bigg( \tfrac{V_\eps''}{ (V_\eps')^2}m(t)^2 q_{\nu, \eps}-\tfrac{V_\eps V_\eps''}{(V_\eps')^2}o'm(t) +\nu \tfrac{(V_\eps'')^2}{(V_\eps')^2}C^3   \\   & \qquad \qquad\qquad \qquad \quad +2m(t)o' +V_\eps o''  - \nu \tfrac{V_\eps'''}{(V_\eps')^4}m(t)^3 + 3 \nu \tfrac{(V_\eps'')^2}{(V_\eps')^2}C^3 \bigg) , 0 \bigg \rbrace.
        \end{split}\label{absder}
    \end{equation}
        Again, two cases have to be distinguished. For this, fix $t \in (0, T)$ such that $m$ is differentiable (this is possible a.e.).
    \begin{itemize}
        \item \textbf{Case 2.1: }($\exists x \in X(t): q_{\nu, \eps}(t, x) \leq \tfrac{o(x)}{2}$) In this case, a bound on $m$ can be obtained immediately without~\cref{absder}: Choose some $x \in X(t)$ such that $q_{\nu, \eps}(t, x) \leq \tfrac{o(x)}{2}$. If we choose $\nu > 0$ as in~\cref{oslq}, then from~\cref{oslq}, 
    \begin{equation}\label{tic}
        o'(x) - \partial_x q_{\nu, \eps}(t, x) \leq  T\| o'' \|_{L^\infty(\mathbb{R})} - \min\lbrace - 2 \| o' \|_{L^\infty(\mathbb{R})}, L_- \rbrace =: C
    \end{equation}
    holds. Here, \(L_- \coloneqq \essinf_{y\in \R}\big(q_0'(y) - o'(y)\big)\).
    Then, we can estimate the following:
    \begin{align*}
    m(t) & = \sup_{y \in \mathbb{R}} \partial_y \big(V_\eps(o(y) -q_{\nu, \eps}(t, y)) \big)  = \partial_x \big(V_\eps(o(x) -q_{\nu, \eps}(t, x))\big) \\
    & = V_\eps'(o(x) -q_{\nu, \eps}(t, x)) (o'(x) -\partial_x q_{\nu, \eps}(t, x)) \\
    & \overset{\text{\eqref{tic}}}{\leq} CV_\eps'(o(x) -q_{\nu, \eps}(t, x))  
    \intertext{where $q_{\nu, \eps}(t, x) \leq \tfrac{o(x)}{2}$ and that $V_\eps$ is decreasing.}
    & \leq C V_\eps'\left(\tfrac{o(x)}{2} \right) \leq C V_\eps'\left(\tfrac{\inf_{y \in \mathbb{R}}o(y)}{2} \right) 
    \end{align*}
    Since the infimum of $o$ is positive per~\cref{as1}, the last expression has a uniform bound in $\eps>0$. Namely,
    \begin{equation}\label{gboound}
    \sup_{y \in \mathbb{R}} \partial_y\big(V_\eps(o(y) -q_{\nu, \eps}(t, y))\big) = \partial_x \big(V_\eps(o(x) -q_{\nu, \eps}(t, x))\big) \leq C \sup_{\eps > 0} \left \lVert V_\eps' \right \rVert_{L^\infty\left(\left[k_0, \infty\right) \right)},
    \end{equation}
    where $k_0 := \tfrac{\inf_{y \in \mathbb{R}}o(y)}{2}$.
    \item \textbf{Case 2.2:} $(q_{\nu, \eps}(t, x) > \tfrac{o(x)}{2}\text{ for all }x \in X(t))$: Here, we estimate~\cref{eq:tbe1,eq:tbe2} (once again in the abbreviated notation) using $q_{\nu, \eps}(t, x) > \tfrac{o(x)}{2}$ and $\tfrac{V_\eps''}{(V_\eps')^2}<0$:
    \begin{align*}
        m'(t) & \leq -2m(t)^2 + V_\eps'\bigg( \tfrac{V_\eps''}{(V_\eps')^2}m(t)^2 \tfrac{o(x)}{2} - \tfrac{V_\eps V_\eps''}{ (V_\eps')^2}o'm(t) + \nu \tfrac{(V_\eps'')^2}{(V_\eps')^2}C^3 + 2m(t)o' \\
       & \qquad + V_\eps o'' - \nu \tfrac{V_\eps'''}{(V_\eps')^4}m(t)^3 + 3 \nu \tfrac{(V_\eps'')^2}{(V_\eps')^2}C^3 \bigg) \\
       \intertext{Recall $- \inf_{y \in \mathbb{R}} o(y) < 0$ and the estimates from the fifth point of~\cref{asH} to control the quotients. We set $v :=  \left \lVert \tfrac{V''}{V'} \right \rVert_{L^\infty\left([0, \infty)] \right)} < \infty$, factor out $(V')^{-1}$ and estimate the other terms roughly:}
       & \leq -2m(t)^2 + V_\eps' \bigg( \tfrac{1}{V'}\big( \tfrac{1}{2} \underbrace{v_{2, 1}^+}_{<0\text{ by~\cref{asH}}} \inf_{y \in \mathbb{R}} o(y) m(t)^2 +v\| o' \|_{L^\infty(\mathbb{R})} | m(t)|\\ 
       & \qquad + 2V' \| o' \|_{L^\infty(\mathbb{R})}| m(t) | \big) +   \| o'' \|_{L^\infty(\mathbb{R})}+  \tfrac{4\nu v^2}{\eps^2}C^3 - \nu \tfrac{V_\eps'''}{(V_\eps')^4}m(t)^3 \bigg)
       \intertext{We additionally require $\nu \leq \tfrac{\eps^2}{4v^2C^3}$ to get remove $\eps$.}
       & \leq -2m(t)^2\\
       & \quad + V_\eps'\bigg(\tfrac{1}{V'}\big(\tfrac{1}{2}v_{2, 1}^+\inf_{y \in \mathbb{R}} o(y) m(t)^2 + \big( v+2\lVert V' \rVert_{L^\infty(\mathbb{R})}  \big) \| o' \|_{L^\infty(\mathbb{R})}| m(t) | \big) \\
       & \qquad \qquad \qquad\qquad + \| o'' \|_{L^\infty(\mathbb{R})} + 1 - \nu \tfrac{V_\eps'''}{(V_\eps')^4}m(t)^3 \bigg)
       \intertext{Let $C_1$ be as in \cref{const1}. If $m(t)\leq C_{1}$, a feasible bound is found. If $m(t) > C_{1}$, then one can compute $\tfrac{1}{2}v_{2, 1}^+\inf_{y \in \mathbb{R}} o(y) m(t)^2 + \big( v+2\lVert V' \rVert_{L^\infty(\mathbb{R})}  \big) \| o' \|_{L^\infty(\mathbb{R})}| m(t) | < \lVert V' \rVert_{L^\infty\left([0, \infty) \right)}\big(-\lVert o''\rVert_{L^\infty(\mathbb{R})} -1\big)$. In particular, $-\nu V_\eps'''(V_\eps')^{-4} \leq 0$. So, we estimate for this case}
       & \leq - 2m(t)^2 + V_\eps'\bigg( \tfrac{\lVert V' \rVert_{L^\infty\left([0, \infty) \right)}}{V'} \big(-\| o'' \|_{L^\infty(\mathbb{R})} - 1 \big) + \| o'' \|_{L^\infty(\mathbb{R})}+1 \bigg) \\
       & \leq -2m(t)^2  \leq 0
    \end{align*}
    Thus, if  a maximal interval $I$ exists such that $m(t) > C_1$ with \(C_{1}\) as in \cref{const1} and $q_{\nu, \eps}(t, x) > \tfrac{o(x)}{2}$ for all $t \in I$, then~\cref{absder} becomes $m'(t) \leq 0$ for a.e.\ $t \in (0, T)$, whence
    \[
    m(t) \leq \max \lbrace m(0), C_{1} \rbrace.
    \]
    If no such interval $I$ exists, then, by continuity, $m$ stays below the maximum $C_{1}$ and the bound computed in~\cref{gboound}.
    \end{itemize}
    Finally, in \textbf{all cases}, we obtain that 
    \[
    m(t) \leq \max \left \lbrace C_{1}, m(0) ,  C_2 \sup_{\eps > 0} \left \lVert V_\eps' \right \rVert_{L^\infty\left(\left[k_0, \infty\right) \right)} \right \rbrace,
    \]
with $C_{2}$ defined as in~\cref{const3}. It remains to prove that $m(0)$ is bounded uniformly in $\nu, \eps$, which is a consequence of
    \begin{align*}
       m(0) = \sup_{x\in \R} \partial_x \left( V_\eps\big( o(x) -q_0(x) \big)\right) &\overset{\cref{asH}, e)}{\leq} \esssup_{y\in \R}(o'(y) - q_0'(y)) \sup_{\eps > 0} \lVert V_\eps' \rVert_{L^\infty\left([c_0, \infty) \right) },
    \end{align*}
    where $c_0 \coloneqq \essinf_{y\in \R} (o(y)-q_0(y))$, which completes the proof.  
\end{proof}

The obtained results allow us to pass to the limit, i.e., $\eps \searrow 0$:
\begin{theorem}[Convergence of $q_\eps$ and $V_\eps(o-q_\eps)$]\label{alllim}
    Let~\cref{as1} and \cref{as2} hold and $q_{ \eps}$ be the solution to \cref{eq:conservation_law_smooth} according to \cref{entropynu} for $\eps \in\R_{>0}$. Then, there exists $q^\star \in C\big([0, T]; L^1_{\text{loc}}(\mathbb{R})\big)$  s.t. along a subsequence $(\eps_{k})_{k\in\N},\ \lim_{k\rightarrow\infty} \eps_{k}=0$, and the following holds:
  \begin{align*}
  q_{\eps_k} &\rightarrow q^\star  \text{ in } C([0, T]; L^1_{\text{loc}}(\mathbb{R})),
  \intertext{ and there exists $V^\star \in L^\infty((0,T)\times \R)$ such that}
V_{\eps_k}(o -q_{\eps_k})  &\overset{*}{\rightharpoonup} V^{\star} \text{ in } L^{\infty}((0,T)\times\R).
\end{align*}
\end{theorem}
\begin{proof}
 This follows from \cref{tvinx,entropynu,tvo,oslq}.
\end{proof}

\begin{remark}[Missing time compactness of \(V_{\eps}\) to obtain strong convergence]
Note that thanks to \cref{oslV}, one obtains by \cref{tvo} total variation estimates of \(V_{\eps}(o-q_{\eps}(t,\cdot))\) in a space uniform in \(\eps\in\R_{>0}\). However, we could not obtain the required (compare \cite[Theorem 3]{Simon1986CompactSI}) time compactness for strong convergence in \(C([0,T];L^{p}_{\text{loc}}(\R)),\  p\in[1,\infty)\), but we are left with a weaker form of convergence: Either the weak-star convergence in \cref{alllim} or,
for every $t \in [0,T]$, there exists $V_t^* \in L^\infty(\R)$ and a sequence $(\eps_k)_{k\in \N}$ with $\lim_{k\rightarrow \infty} \eps_k = 0$ such that
    \begin{align*}
  \lim_{k\rightarrow \infty} \| V_{\eps_k}(o(\cdot) -q_{\eps_k}(t, \cdot)) - V^*_t \|_{L^1_{\text{loc}}(\mathbb{R})}  &= 0.
  \end{align*}
  Indeed, $V_t^*$ is OSL-bounded from above (see~\cref{oslV}).
\end{remark}
One may wonder what the velocity \(V_{\eps}(o-q_{\eps})\) and solution \(q_{\eps}\) satisfy in the limit. This is characterized in the following:
\begin{lemma}[Dynamics of the solution and velocity in the limit \(\eps\rightarrow 0\)]\label{lem:limit_dynamics}
The sequence \(q_{\eps_{k}},\) \(V_{\eps_{k}}(o-q_{\eps_{k}})\) of \cref{alllim} satisfies in the limit \(\eps\rightarrow 0\) the following weak form of a conservation law for \(\varphi\in C^{1}_{\text{c}}((-42,T)\times\R)\):
    \begin{align}
        \int^T_0 \!\!\!\int_{\mathbb{R}} q^\star(t, x) \partial_t \varphi(t, x) + V^\star(t,x)q^\star(t,x)\partial_x  \varphi(t, x)\dd x \dd t 
        & = -\int_\R q_0(x) \varphi(0,x) \dd x \label{dsol}
    \end{align}
\end{lemma}
\begin{proof}
    This is a direct consequence of \cref{alllim} together with the fact that the product of a strongly and a weakly-* convergent sequence converges weak-*.
\end{proof}
\begin{remark}[Discontinuous conservation laws/relation to \cref{lem:limit_dynamics}]\label{rem:discontinuous_conservation_law}
Looking back to the regularized obstacle problem in \cref{eq:conservation_law_smooth}, this equation converges to a discontinuous (in the solution and spatial variable (!)) conservation law as the velocity converges to a discontinuous function, i.e., 
\[
\lim_{\eps\rightarrow 0}V_{\eps}(o(x)-q_{\eps}(t,x)) \in \begin{cases} \{1\} &\text{if } q^{*}(t,x)< o(x)\\
[0,1]&\text{if } q^{*}(t,x)=o(x),
\end{cases} \qquad(t,x)\in(0,T)\times\R
\]
and thus, we obtain
        \begin{align}
        \int^T_0 \!\!\!\int_{\mathbb{R}} q^\star(t, x) \partial_t \varphi(t, x) + H(q^\star(t,x))q^\star(t,x)\partial_x  \varphi(t, x)\dd x \dd t 
        & = -\int_\R q_0(x) \varphi(0,x) \dd x \label{eq:discont_prob},
    \end{align}
    where \(\varphi\) is a test function as in \cref{lem:limit_dynamics} and
    $H$ is the Heaviside function with a value at zero that is undetermined, i.e., can attain any value in $[0,1]$.  
One may wonder whether applying results on the uniqueness of these discontinuous conservation laws is possible \cite{Bulek2011scalar,Bulek2017unified,Carrillo2003conservation,Martin2008convergence,Dias2005approximation,Dias2004riemann}; however, the discontinuous conservation law considered here does not seem to fit into the required frameworks.

\end{remark}

\section{Characterization of the limit \texorpdfstring{\(\eps\searrow 0\)}{} in specific cases}\label{sec:characterization_limit}
This section presents characterizations of the solutions approached in \cref{alllim} under additional assumptions on its smoothness and more. They serve to build an intuition on what happens in this limit as well as how the dynamics behave at the obstacle.

Suppose we have a solution $q\coloneqq q^\star$ as in~\cref{dsol} that is piecewise smooth in the sense that it is smooth outside of a finite number of curves in $(t,x)$, across which $q$ or $\partial_x q$ has jump discontinuities. We take a representative that is also piecewise smooth, such that $q$ and $\partial_x q$ have well-defined traces at every point of discontinuity and the coincidence set $E\subset [0,T) \times \R$, $E= \{(t,x) \in [0, T) \times \mathbb{R}: q(t,x) = o(x)\}$ is of the form
\begin{equation*}
  E = \{(t,x) \in [0, T) \times \mathbb{R}: \gamma_L(t) \leq x \leq \gamma_R(t)\},
\end{equation*}
for two continuous piecewise smooth curves $\gamma_L \leqq \gamma_R$ on \([0,T]\) with $\gamma_R$ non-decreasing. 
Suppose there exists a well-defined first collision time. We can set it without loss of generality to be $t=0$ (this is a contradiction to \cref{as1} as we postulated \(q_{0}-o<0\ \text{on } \R\); however, thanks to the semi-group property of any conservation law, we can shift time accordingly).

Then, the following result relying on the Rankine--Hugoniot condition \cite{Rankine1870thermodynamic,hugoniot1887memoir} gives a characterization of the boundary curves of the coincidence region $E$ in terms of the traces of $q$ and $\partial_x q$ from the outside of $E$:
\begin{theorem}\label{thm:Limit_V_Rankine-Hugoniot}
  Assume that $q_0$ and $o$ adhere to~\cref{as1}, $o$ is strictly decreasing on some interval $(-\infty,x_0)$, and the limits for \(t\in[0,T]\),
  \begin{align*}
  q^L(t)\coloneqq \lim_{y\nearrow \gamma_L(t)}q(t,y), \quad  q^R(t)\coloneqq \lim_{y\searrow \gamma_R(t)} q(t,y),  \quad q_x^R(t) \coloneqq  \lim_{y\searrow \gamma_R(t)} \partial_xq(t,y),
\end{align*}
exist for a.e.\ point along the curves $\gamma_L,\gamma_R$. Suppose further that $|\partial_x^2 q|$ is uniformly bounded on the (open) set $E^c$.
  Under the previous assumptions, we have the following:
  \begin{enumerate}
  \item The velocity field $(t,x)\mapsto V^*(t,x)$ as in \cref{lem:limit_dynamics} (and thus the speed of the characteristics) is given by
    \begin{equation}\label{eq:1}
      V^\star(t,x) =\begin{cases}
        1, & (t,x) \not\in E,
        \\
        \tfrac{o(\gamma_R(t))}{o(x)}, & (t,x) \in E.
      \end{cases}
    \end{equation}
  \item \(\forall t\in[0,T)\), $o(\gamma_L(t)) - q^L(t) >0$ holds, and
    \begin{equation}\label{eq:112}
            \gamma_L'(t) = \tfrac{o(\gamma_R(t))- q^L}{o(\gamma_L(t))- q^L}.
          \end{equation}
  \item In intervals where $t\mapsto\gamma_R(t)$ is differentiable, either $\gamma_R'(t) = 0$, or if $q$ is smooth at $x=\gamma_R(t)$,
      \begin{equation}\label{eq:113}
          \gamma_R'(t) = \tfrac{q_x^R(t)}{q_x^R(t) - o'(\gamma_R(t))}>1,
        \end{equation}
       which, if $q$ has a jump discontinuity at $x=\gamma_R(t)$, results in
    \begin{equation}\label{eq:2}
            \gamma_R'(t) = 1.
          \end{equation}
 \end{enumerate}
  In particular, $q$ has a discontinuity along $\gamma_L$. When $\gamma_L'(t) \le 0$, then $\gamma_L'(t)$ is the speed of a backward \emph{congestion shock} originated by the collision with the obstacle.
\end{theorem}

  \proof
  Consider a sufficiently small open ball $B$ centered on a point $\gamma_L(t)$ where $\gamma_L$ is smooth. Let $\Gamma_L=\big \{(t, y) \in [0, T] \times \mathbb{R}: (t, y) = (t,\gamma_L(t))\big \}$ and set $B = B_E \cup (\Gamma_L\cap B) \cup B_{E^c}$ with obvious notation (see~\cref{BE}).
   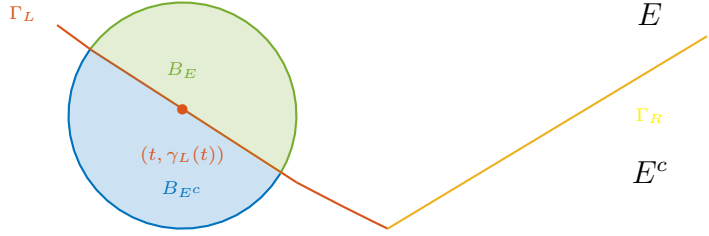
\begin{figure}[h!]
      		\begin{center}
			\begin{tikzpicture}[scale = 1.5]
    			\node[at={(-4.5, -0.3)}, matlaborange]{\footnotesize $\Gamma_L$};
                \node[at={(1, -1.2)}, yellow]{\footnotesize $\Gamma_R$};
                \node[at={(-3.1, -1.85)}, matlabblue]{\footnotesize $B_{E^c}$};
                \node[at={(-3.1, -0.8)}, matlabgreen]{\footnotesize $B_{E}$};
                \node[at={(1, -0.3)}, black]{\large $E$};
                \node[at={(1, -1.7)}, black]{\large $E^c$};
                \draw [matlabgreen,thick,domain=-31:144, name path = E] plot ({cos(\x)-3.1}, {sin(\x)-1.2});
                \draw [matlabblue,thick,domain=144:329, name path = F] plot ({cos(\x)-3.1}, {sin(\x)-1.2});
                \draw [matlaborange, thick] plot [smooth] coordinates {(-4.2, -0.4) (-3.9, -0.62)};
                \draw [matlaborange, thick, name path = B] plot [smooth] coordinates {(-3.9, -0.62) (-2.1, -1.79)};
                \draw [matlaborange, thick] plot [smooth] coordinates {(-2.1, -1.79) (-1.7, -2) (-1.3, -2.2)};
                \draw [matlabyellow, thick, name path = C] plot [smooth] coordinates {(-1.3, -2.2) (1.5, -0.5)};
                \tikzfillbetween[of=B and E]{matlabgreen, opacity=0.2};
                \tikzfillbetween[of=B and F]{matlabblue, opacity=0.2};
                \node[at={(-3.1, -1.15)}, matlaborange]{\textbullet};
                \node[at={(-3.1, -1.55)}, matlaborange]{\footnotesize $(t, \gamma_L(t))$};
			\end{tikzpicture}
		\end{center}
        \caption{Illustration of the sets $\Gamma_L=\{(t,\gamma_L(t))\}$, $B_{E^c}$ and $B_E$. $E$ is the set ``above'' the green and orange curves. $E^c$ denotes is its complement.}
      \label{BE}
  \end{figure}
   
Take a test function $\varphi$ supported in $B$. Since $q$ is a distributional solution of the conservation law
  $
  q_t + (V^*q)_x =0
  $
  in $(0,T)\times \R,$ we have (here, for the sake of brevity, we omit the arguments of the functions)
  \begin{align}
    0 &= \int_B q \partial_t \varphi +V^* q \partial_x \varphi~\dd (t,x)\notag \\
     & = \int_{B_E}  q \varphi_t + V^* q \partial_x \varphi~ \dd (t, x) + \int_{B_{E^c}} q \varphi_t + V^* q \partial_x \varphi~ \dd (t, x) \notag
    \intertext{as $q\equiv o \text{ on } E,\, V^* \equiv 1 \text{ on }E^c$, and we obtain}
    & = \int_{B_E} o\partial_t \varphi+ V^* o \partial_x \varphi \dd (t, x)   + \int_{B_{E^c}} q \partial_t \varphi + q \partial_x \varphi \dd (t, x).\label{eq:222}
  \end{align}
  Let us look at each of the two integrals in turn. First, note that in the coincidence set $E$, the obstacle $o$ is a distributional solution of 
  \begin{equation}\label{distrobs}
  \partial_t o(x) + \partial_x(V^*(t,x)o(x)) = 0 \ \forall x \in E.
  \end{equation}
Since $o$ does not depend on $t$, we conclude that $(t,x) \rightarrow \partial_x(V^*(t,x)o(x))$ is the zero distribution on $E$. Therefore, owing to distribution theory, the function $\big((t,x) \mapsto V^*(t,x)o(x) \big) \in L^\infty(E)$ is constant in $x$ for almost all $t$. Thus, $V^*(t,\cdot)o(\cdot) \equiv c_t$ holds a.e. on $ (\gamma_L(t),\gamma_R(t))$ for a.e.~$t$.

  To identify $c_t$, we recall that $V_\eps(o-q_\eps)$ is OSL continuous from below in the spatial variable (see~\cref{oslV} and apply the limit) uniformly in $\eps$. Let us see that, actually,
  \begin{equation*}
    c_t = o(\gamma_R(t)).
  \end{equation*}
  Let $L>0$ be this Lipschitz constant. Writing $x_1 = \gamma_R(t)$, observe that for a.e.~$t$ for which $\gamma_{L,R}(t)$ are defined, and $\eps > 0$:
\[
V_\eps\bigg(o\big(\underbrace{x_1+\tfrac{1}{k}}_{\notin E} \big) - q_\eps\big(t, x_1+\tfrac{1}{k}\big)\bigg) - \tfrac{L}{k} \leq V_\eps\big(o(x_1) -q_\eps(t, x_1) \big) \leq 1.
\]
As $\eps \searrow 0$, we obtain that the left-hand side converges to $1 - \tfrac{L}{k}$ since $x_1 + \tfrac{1}{k} \notin E$ (see the definition of $V_\eps$), whereas $V_\eps\big( o(x_1) -q_\eps(t, x_1)\big) \rightarrow V^*(t, x_1)$ as $\eps \searrow 0$. So, in total,
\[
1 - \tfrac{L}{k} \leq V^*(t, x_1) \leq 1
\]
for all $k \in \mathbb{N}$, whence $V^*(t, x_1) = 1$ for a.e. $t \in [t_0, t_1]$. Thus, plug $x = x_1$ into $V^*(t,x)o(x) = c_t$ to obtain $c_t = o(x_1)=o(\gamma_R(t))$. This proves the representation \cref{eq:1}.

 Going back to \cref{eq:222}, we deduce for the first integral, using the divergence theorem  (let $\sigma$ denote the $1D$ surface measure),
  \begin{align*}
    &\int_{B_E} o\partial_t \varphi + V^* o \partial_x \varphi ~\mathrm{d}(t, x) \\
    & = \int_{B_E} o \partial_t \varphi +  (o\circ \gamma_R) \partial_x \varphi ~\mathrm{d}(t, x)
    \\
    & \overset{\mathrm{div}-\text{Thm.}}{=} \int_{B_E} \mathrm{div}_{t,x}\big(o \varphi, ( o \circ \gamma_R) \varphi \big)^\top ~\mathrm{d}(t, x)
    \\
    & = \int_{\Gamma_L \cap B} (o,  o \circ \gamma_R)^\top \cdot (-\gamma_L',1)^\top \varphi~\mathrm{d}\sigma
    \\
    & = \int_{\Gamma_L \cap B} [-o\gamma_L+ o\circ \gamma_R] \varphi ~\mathrm{d}\sigma.
  \end{align*}

  For the second integral in \cref{eq:222}, we find again with the divergence theorem and the fact that on $E^c$, $q$ satisfy the transport equation,
  \begin{align*}
   &\int_{B_{E^c}} q \partial_t \varphi+ q \partial_x \varphi ~\mathrm{d}(t, x) \\
   &= -\int_{B_{E^c}} (\partial_t q + \partial_x q) \varphi~\mathrm{d}(t, x)
    + \int_{\Gamma_L \cap B} (q,q)\cdot (\gamma_L',-1)  \varphi \dd \sigma
    \\
    &= 0+  \int_{\Gamma_L \cap B} ( q^L \gamma_L' - q^L)  \varphi \dd\sigma.
  \end{align*}
  Therefore, the following holds:
  \begin{equation*}
    \int_{\Gamma_L\cap B} (-o\gamma_L'+ (o\circ \gamma_R) \varphi \dd\sigma = - \int_{\Gamma_L} ( q^L \gamma_L' - q^L)  \varphi \dd\sigma.
  \end{equation*}  
  Since $\varphi$ is arbitrary, we conclude that
  \begin{equation}\label{eq:555}
    \aligned
    -o(\gamma_L(t))\gamma_L'(t)+ o(\gamma_R(t)) = - q^L \gamma_L'(t) + q^L
    \endaligned
  \end{equation}
  for a.e. $t \in (0, T)$ exploiting the given structure of $\Gamma_L$ and the fact that the center of $B$ is arbitrary.\\
Now, suppose that $o(\gamma_L(t)) = q^L$ and $\gamma_L(t)$ is strictly smaller than $\gamma_R(t)$. Then, from the strict monotonicity of $o$ and \cref{eq:555}, we have $o(\gamma_L(t)) < o(\gamma_R(t))=q^L$, which is absurd. Therefore, $o(\gamma_L(t)) - q^L \neq 0$, and \cref{eq:112} follows by rearranging.

\cref{eq:2,eq:113} remain to be proven. Suppose that $q$ has a jump discontinuity at $x=\gamma_R(t)$. Then, by assumption, there exists the trace $q^R,$ and, necessarily, $q^R < o(\gamma_R(t))$. The proof of \cref{eq:2} is entirely analogous to the proof of \cref{eq:112}, but instead of \cref{eq:555}, we obtain
  \begin{equation*}
    \aligned
    &-o(\gamma_R(t))\gamma_R{'}(t)+ o(\gamma_R(t)) = - q^R \gamma_R{'}(t) + q^R
    \\
    & \implies \gamma_R{'}(t) (o(\gamma_R(t)) - q^R) = o(\gamma_R(t)) - q^R  ,  
    \endaligned
  \end{equation*}
  and thus, $\gamma_R'(t) =1$ since $q^R - o(\gamma_R(t)) \neq 0$.

  Suppose now that $q$ is continuous at $x=\gamma_R(t)$, $q_x$ has a trace from the right, $q_x^R$,  and $\gamma_R'(t) >0$. Then, $\frac{\dd}{\dd t}o(\gamma_R(t))<0$, which implies that $o'(\gamma_R(t)) <0$.
  Observe that necessarily, $q_x^R < o'(\gamma_R(t)) <0$. The formal calculation to show \cref{eq:113} is as follows. Compute the derivative of $o$ along the curve $\gamma_R$ to find $\frac{\dd}{\dd t} o(\gamma_R(t)) = o'(\gamma_R(t)) \gamma_R'(t)$. But, along $\gamma_R$, $o=q$ holds; therefore, using ~\cref{distrobs} for $q$ on $E^c$,
  \begin{equation}\label{eq:888}
    \aligned
    \frac{\dd}{\dd t} o(\gamma_R(t)) &= \frac{\dd}{\dd t} q(t,\gamma_R(t)) = \partial_t q(t,\gamma_R(t)) + \gamma_R'(t) \partial_x(t,\gamma_R(t))
    \\
    & = -\partial_2 q(t,\gamma_R(t)) + \gamma_R'(t) \partial_2 q(t,\gamma_R(t))
    \\
    & = -q_x^R + \gamma_R'(t) q_x^R,
    \endaligned
  \end{equation}
  giving \cref{eq:113}. However, while $t\mapsto q(t,\gamma_R(t))$ is differentiable, $q$ is not differentiable at $(t,\gamma_R(t))$, and so, the calculation is not justified. To circumvent this problem, a more careful computation is needed, which we now provide.

  Let $\lambda >0$ and consider the quantity
  \begin{equation*}
    S_{t,\lambda} \coloneqq\frac{q(t,\gamma_R(t)) - q(t-\lambda,\gamma_R(t-\lambda))}{\lambda}.
  \end{equation*}
  On the one hand, we have
  \begin{equation}\label{eq:777}
    \aligned
     \lim_{\lambda \searrow 0} S_{t,\lambda} = \frac{\dd}{\dd t}q(t,\gamma_R(t)).
    \endaligned
  \end{equation}
  On the other hand, we find
  \begin{equation}\label{eq:666}
    \aligned
    S_{t,\lambda} = \frac{q(t,\gamma_R(t)) - q(t-\lambda,\gamma_R(t))}{\lambda} + \frac{q(t-\lambda,\gamma_R(t)) - q(t-\lambda,\gamma_R(t-\lambda))}{\lambda}.
    \endaligned
  \end{equation}
  The first term gives
  \begin{align*}
    \frac{q(t,\gamma_R(t)) - q(t-\lambda,\gamma_R(t))}{\lambda} & = -\frac1{\lambda} \int_0^\lambda \frac{\dd}{\dd s}\big(q(t-s,\gamma_R(t))\big) \dd s
    \\
    & = \frac1\lambda \int_0^\lambda \partial_t q(t-s,\gamma_R(t)) \dd s
    \\
    &= - \frac1\lambda \int_0^\lambda \partial_2q(t-s,\gamma_R(t)) \dd s \to - q_x^R
  \end{align*}
  as $\lambda \to 0$, from our smoothness assumptions on $q$ from outside of $E$. Note that we can use the transport equation for $q$ since $(t-s,\gamma_R(t))$ is \emph{outside} the coincidence region $E$. 
  For the second term in \cref{eq:666}, we find 
  \begin{align*}
    \frac{q(t-\lambda,\gamma_R(t)) - q(t-\lambda,\gamma_R(t-\lambda))}{\lambda} & = -\frac1{\lambda} \int_0^\lambda \frac{\dd}{\dd s}\big(q(t-\lambda,\gamma_R(t-s))\big) \dd s
    \\
    &\hspace{-25mm} =  \frac1\lambda \int_0^\lambda \partial_2q(t-\lambda,\gamma_R(t-s))\gamma_R'(t-s) \dd s
    \\
    & \hspace{-25mm} = \frac1\lambda \int_0^\lambda \partial_2q(t-\lambda,\gamma_R(t))\gamma_R'(t-s) \dd s
    \\
    &\hspace{-25mm}\quad + \frac1\lambda\! \int_0^\lambda\!\!\!\! \big[  \partial_2q(t-\lambda,\gamma_R(t-s)) - \partial_2q(t-\lambda,\gamma_R(t)) \big]\gamma_R'(t-s)  \dd s
    \\
    & \hspace{-25mm} = \partial_2q(t-\lambda,\gamma_R(t))\frac1\lambda \int_0^\lambda \gamma_R'(t-s) \dd s
    \\
    &\hspace{-25mm}\quad - \frac1\lambda \int_0^\lambda\int_0^s  \frac{\dd}{\dd r} \big(\partial_2q(t-\lambda,\gamma_R(t-r))\big)\gamma_R'(t-s) \dd r \dd s
    \\
    & \hspace{-25mm} = \partial_2q(t-\lambda,\gamma_R(t))\frac1\lambda \int_0^\lambda \gamma_R'(t-s) \dd s
    \\
    &\hspace{-25mm}\quad + \frac1\lambda \int_0^\lambda\int_0^s   \partial_2^2q(t-\lambda,\gamma_R(t-r))\gamma_R'(t-r)\gamma_R'(t-s) \dd r \dd s.
  \end{align*}
  Now, clearly the first term converges to $q_x^R \gamma_R'(t)$ when $\lambda\to 0$, while the second can be easily bounded by $\lambda \|\partial_x^2 q\|_{L^\infty(E^c)} \|\gamma_R'\|_{L^\infty([0, T])}$, which also vanishes in the limit $\lambda\to 0$.

  Therefore, $\lim\limits_{\lambda\searrow 0}S_{t,\lambda} = -q_x^R + \gamma_R'(t) q_x^R,$ which together with \cref{eq:777} gives the rigorous analogue of \cref{eq:888}.
  Thus, we find that
  \begin{align*}
     o'(\gamma_R(t))\gamma_R'(t) =  -q_x^R + \gamma_R'(t) q_x^R
   \ \iff  \ \gamma_R'(t) (o'(\gamma_R(t)) - q_x^R ) = -q_x^R.
  \end{align*}
  Finally, recalling that $q_x^R < o'(\gamma_R(t)) <0,$ we obtain \cref{eq:113} by rearranging.  
This completes the proof.
  
\endproof

\section{Motivation of velocity \texorpdfstring{for \(\eps\searrow 0\)}{in the limit} by optimization}\label{sec:motivation_optimization}
The result from~\cref{thm:Limit_V_Rankine-Hugoniot} can also be \emph{formally} motivated from an optimization perspective. We can, in fact, demonstrate that the velocity $V^\star$, found in the previous chapter, is maximal in an $L^1$ sense for each time.

To this end, let us first assume that the density $q$ moves with a space- and time-dependent velocity $v:(0,T)\times \mathbb{R} \rightarrow [0, \infty)$ (the choice of this dependency becomes clear later) while obeying the scalar conservation law
\begin{equation}\label{scl}
\partial_t q(t,x) + \partial_x (v(t,x)q(t, x)) = 0, \quad q(0, x) = q_0(x)\text{  on  }(0, T) \times \mathbb{R}.
\end{equation}
The conservation laws' characteristics emanating from \((t,x)\in(0,T)\times\R\), i.e.,\ $\xi_v[t, x]:[0, T] \rightarrow \mathbb{R}$ are  given as a solution of the ODE (see, e.g., \cite[Sec. 2]{keimernonlocalbalance2017}):
\begin{equation}
\partial_s \xi_v[t, x](s) = v\big(s,\xi_v[t, x](s)\big), \quad \xi_v[t, x](t) = x,\ s\in[0,T]\label{eq:characteristics}
\end{equation}
We assume that $v$ is OSL from below in the spatial variable and essentially bounded in \(L^{\infty}((0,T)\times\R\). According to~\cite{keimernonlocalbalance2017,Bouchut1998onedimensional}, the solution to~\cref{scl} satisfies
\[
q(t, x) \coloneqq q\big(t_0,\xi_v[t, x](t_0)\big) \partial_2 \xi_v[t, x](t_0)
\]
for all $(t, x) \in [0, T] \times \mathbb{R}$ and $t_0 \in [0, T]$. 
The previously mentioned solution formula makes the semi-group property (in time) of the dynamics visible.

We then define the $L^2$-obstacle violation by introducing the function $V:[0, T] \rightarrow \mathbb{R}$ for all \(t \in [0, T]\) and choose a \(t_{0}\in (0,T)\):
\begin{align}
V(t) & \coloneqq  \int_{\mathbb{R}} \bigg( q\big(t_0,\xi_v[t, x](t_0)\big) \partial_2 \xi_v[t, x](t_0) - o'(x)  \big)^+\bigg)^2~\mathrm{d}x\label{eq:defi_V}
\intertext{Performing a substitution according to \cref{eq:characteristics}, i.e.,\ \(y \coloneqq \xi_v[t, x](t_0)\), we obtain}
&=\int_{\mathbb{R}} \bigg(\big( q(t_0,y) -o(\xi_v[t_0, y](t)) \partial_2 \xi_v[t_0, y](t) \big)^+\bigg)^2~\mathrm{d}y,\notag
\end{align}
with \((\cdot)^{+}\coloneqq\max\{\cdot,0\}\). Thanks to the previously assumed regularity, we can compute the time-derivative and obtain the following by applying the chain rule for \(t\in[0,T]\):
\begin{align}
    V'(t) &= 2\int_{\mathbb{R}} \big( q(t_0,y) -o(\xi_v[t_0, y](t)) \partial_2 \xi_v[t_0, y](t) \big)^+ \big(-o(\xi_v[t_0, y](t))\partial_{23}\xi_v[t_0, y](t) \notag \\
    & \qquad - o'(\xi_v[t_0, y](t)) \partial_3 \xi_v[t_0, y](t) \partial_2 \xi_v[t_0, y](t)\big) ~\mathrm{d}y \notag, \\
    \intertext{and substituting  $\partial_3 \xi_v[t_0, y](t) = v(\xi_v[t_0,y](t))$  according to \cref{eq:characteristics},}
    &  = -2\int_{\mathbb{R}} \underbrace{\big( q(t_0,y) -o(\xi_v[t_0, y](t)) \partial_2 \xi_v[t_0, y](t) \big)^+}_{\geq 0}\notag \\
    & \qquad \cdot \big(o(\xi_v[t_0, y](t))v'(\xi_v[t_0,y](t))  +\partial_y (o(\xi_v[t_0, y](t))) v(\xi_v[t_0,y](t))\big)~\mathrm{d}y. \label{viol}
\end{align}
This yields, in particular (recalling again the properties of the characteristics in \cref{eq:characteristics}),
\begin{equation}
V'(t_{0})= -2\int_{\mathbb{R}} \underbrace{\big( q(t_0,y) -o(y) \big)^+}_{\geq 0}\big(o(y)v'(y)  + o'(y) v(y)\big)~\mathrm{d}y. \label{eq:V_diff}
\end{equation}
On an abstract level, we have the function \(V\) measuring the violation of the obstacle. We also have (the initial datum respects the obstacle) that \(V(0)=0\). If we can manage to show that
\begin{equation}
V'(t)\leq 0 \ \forall t\in[0,T]:\ V(t)>0,\label{eq:V'leq0}
\end{equation}
we know that the obstacle is never violated, i.e, \(V\equiv 0\).

Thus, assume there exists \(t\in(0,T]\) such that \(V(t)>0\). Then, there exists \(E_{>}(t)\subset\R\) measurable with Lebesgue-measure greater zero so that
\[
q(t,y)-o(y)>0\ \forall y\in E_{>0}(t).
\]
Recalling \cref{eq:V_diff}, we have 
\[
V'(t)=-2\int_{E_{>}(t)} \underbrace{\big( q(t_0,y) -o(y) \big)^+}_{>0}\big(o(y)v'(y)  + o'(y) v(y)\big)~\mathrm{d}y,
\]
and so, we obtain by postulating \cref{eq:V'leq0} in a weak sense that
\[
o(y)v'(y)  + o'(y) v(y) \geq 0\ \forall y\in E_{>}(t).
\]
If we keep in mind that we want the velocity at time \(t\in[0,T]\) to be maximal in some topology, we obtain the optimization problem (in \(L^{1}\))
\begin{equation}
\begin{aligned}
    \max_{v\in L^{\infty}(\R;[0,1])}&\int_{\R} v(x)-1\dd x\\
        &\text{subject to}\\
        &\tfrac{\dd}{\dd y}\big(o(y)v(y)\big)\geq 0\ &&\forall y\in E_{>}(t),
\end{aligned}
\label{eq:optimality_system}
\end{equation}
where the inequality constraint is meant distributionally.

For sets \(E_{>}(t)\) that are irregular (e.g., Cantor sets),  the authors are unccertain on how to characterize solutions to the \cref{eq:optimality_system}, but if this could be established, it may even present a more direct way to ``solve the obstacle problem.''
However, assuming that the set \(E_{>}(t)\) is a union of disjoint intervals \((I_{j})_{j\in J}\subset \R\) for an index set \(J\), we can solve on each on these intervals \cref{eq:optimality_system}. %
\begin{theorem}[The solution to \cref{eq:optimality_system}]\label{maxcon}
For \(t\in[0,T]\), the optimality system in \cref{eq:optimality_system} with the additional assumption that, for an index set \(J\) and a set of disjoint intervals \((I_{j})_{j\in J}\),
\[
 E_{>}(t)=\bigcup_{j\in J}I_{j},\quad I_{j}\cap I_{k} =\emptyset \qquad\forall j\in J,\ \forall k\in J\setminus\{j\}
 \]
 admits a unique solution
 \[
v_{t}(x)=\begin{cases}
    \frac{o(\sup(I_{j}))}{o(x)}, &x\in I_{j}, j\in J\\
    1&\text{else},
\end{cases}
\qquad \forall x\in\R.
 \]
\end{theorem}    
This characterizes the velocity at time \(t\in[0,T]\) so that one can see the solution to the obstacle as the solution of the conservation law
\[
\partial_{t}q(t,x)+ \partial_{x}\big( v_{t}(x)q(t,x)\big)=0.
\]
The choice of considering the solution on characteristics \cref{eq:defi_V} is due to our requiring the expression to be time-sensitive concerning the obstacle violation and differentiable.

Remarkably, under the assumptions in \cref{maxcon}, the characterization of the velocity coincides with the results in \cref{sec:characterization_limit}, namely, \cref{eq:1}. This underlines the reasonability of the proposed approach, as in a regular enough setting, all presented approaches (regularization in \(V_{\eps}\) in \cref{sec:Lipschitz}, characterization by Rankine--Hugoniot type approach in \cref{sec:characterization_limit}, and the optimization approach in this \cref{sec:motivation_optimization}) coincide.

\section{Visualization by means of a tailored Godunov scheme}~\label{sec:numerics}
In this section, we conduct numerical studies to ``validate'' the theoretical results about properties of the solution to~\cref{eq:conservation_law_smooth}, i.e.,
\[
\partial_t q_\eps(t, x) + \partial_x \big(V_\eps(o_i(x)-q_\eps(t, x))q_\eps(t, x) \big) = 0
\]
on $(0, T) \times \mathbb{R}$ with initial condition $q(0, \cdot) = q_0^i$, $i \in \lbrace 1, 2, 3 \rbrace$. Here, \emph{if not stated otherwise}, $V_\eps$ is as in~\cref{ex:V}, i.e., 
\( V_\eps(s) = 1-\exp\big(\tfrac{s}{\eps}\big),\ s\in\R_{\geq0}\), and $o_1, o_2, o_3:\mathbb{R} \rightarrow \mathbb{R}$,
\begin{align}
    &\bullet \ o_1(x) \coloneqq  -\exp(-x^2) + \tfrac{3}{2} \notag \\
    &\bullet \ o_2(x) \coloneqq -\tfrac{2}{3}\exp\left(-20\left(x+\tfrac{1}{2}\right)^2\right)-\exp\Big(-20x^2\Big)+\tfrac{3}{2}  \label{eq:example_defi_obstacle} \\
    &\bullet \ o_3(x) \coloneqq \min \left \lbrace \max \left\lbrace  |x|, \tfrac{1}{2} \right \rbrace, \tfrac{3}{2} \right \rbrace \notag
\intertext{for all $x \in \mathbb{R}$ and the corresponding initial data $q_0^1, q_0^2, q_0^3:\mathbb{R} \rightarrow \mathbb{R}$},
    &\bullet \ q_0^1(x) \coloneqq \chi_{\left[-1.5, -1\right]}(x)-x\chi_{\left(-1, 0 \right]} \notag \\
    &\bullet \ q_0^2(x) \coloneqq \max \lbrace 0, 1-3(x+1)^2 \rbrace^2 + \max \lbrace 0, 1-3(x-1)^2 \rbrace^2 \label{eq:example_defi_initial-datum}\\
    &\bullet \ q_0^3(x) \coloneqq \chi_{[-1.5, -1]}(x) \notag
\end{align}
for all $x \in \mathbb{R}$. Of note, $o_3$ (due to nondifferentiability) and $q_0^3$ (due to lack of Lipschitz bounds) are \emph{not} in accordance with~\cref{as1}.

To simulate~\cref{eq:conservation_law_smooth}, we employ a space-dependent Godunov method as laid out in \cite[(34)]{Zhang2003} but use an adaptive step size such that a Courant-Friedrichs-Lévy-type condition (for non--space-dependent flux case, see~\cite[(13.11)]{LeVeque1992}) holds in every time step.\\
For the spatial step size, we chose $\Delta x = 10^{-4}$. The adapted temporal step size then leads to several million time steps in every instance.

\subsection{Visualization of \texorpdfstring{$q_\eps$}{q} for different scenarios}\label{qepssim}
In~\cref{qepssim}, we let $\eps = \tfrac{1}{1024}$ and observe the behavior of $q_\eps$ for different obstacles and initial data.
Choosing $o_1$ as the obstacle and $q_0^1$ as the initial datum, we obtain the result in~\cref{q01o1}. Density accumulates ``in front of'' the obstacle, which can be interpreted as a ``traffic jam''. Also, density uses almost all the possible space under $o_1$. In \cref{3d}, we offer some more perspectives that also indicate very clearly that the obstacle is obeyed.

\begin{figure}
    \centering
    \includegraphics[scale=0.5,clip,trim=0 27 12 0 0]{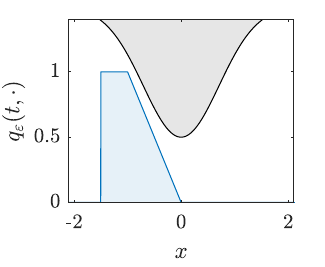}
    \includegraphics[scale=0.5,clip,trim=32 27 12 0 0]{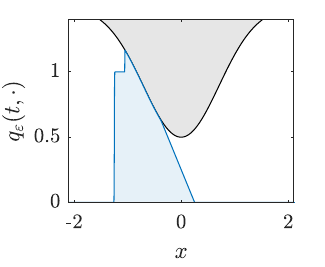}
    \includegraphics[scale=0.5,clip,trim=32 27 12 0 0]{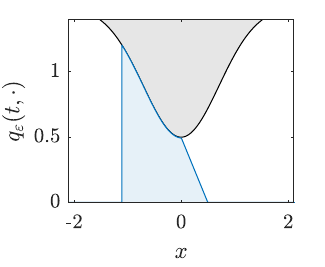}
    \includegraphics[scale=0.5,clip,trim=32 27 12 0 0]{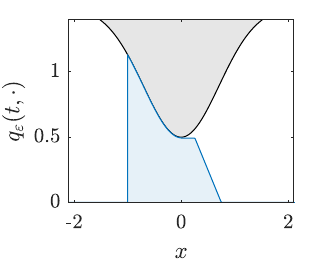}

        \includegraphics[scale=0.5,clip,trim=0 0 12 5]{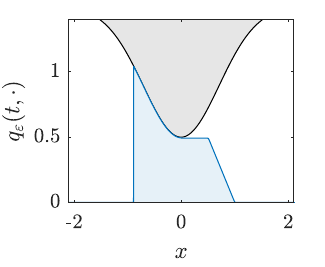}
    \includegraphics[scale=0.5,clip,trim=32 0 12 5]{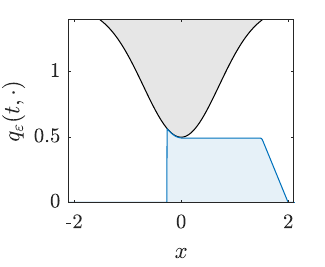}
    \includegraphics[scale=0.5,clip,trim=32 0 12 5]{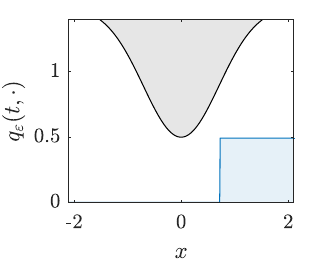}
    \includegraphics[scale=0.5,clip,trim=32 0 12 5]{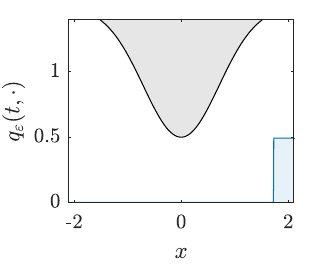}
    \caption{Plots of $o_1$ and \textcolor{matlabblue}{$q_\eps$} for $\eps = 2^{-10}$ and initial datum $q_{0}^{(1)}$ specified in \cref{eq:example_defi_obstacle,eq:example_defi_initial-datum}. The chosen time points are $t=0$, $t=0.25$, $t=0.5$, $t=0.75$ (top, from left to right) and $t=1$, $t=2$, $t=3$, $t=4$ (bottom, from left to right).}
    \label{q01o1}
\end{figure}
\begin{figure}
    \centering
    {\includegraphics[scale=0.5,clip,trim=11 0 25 5]{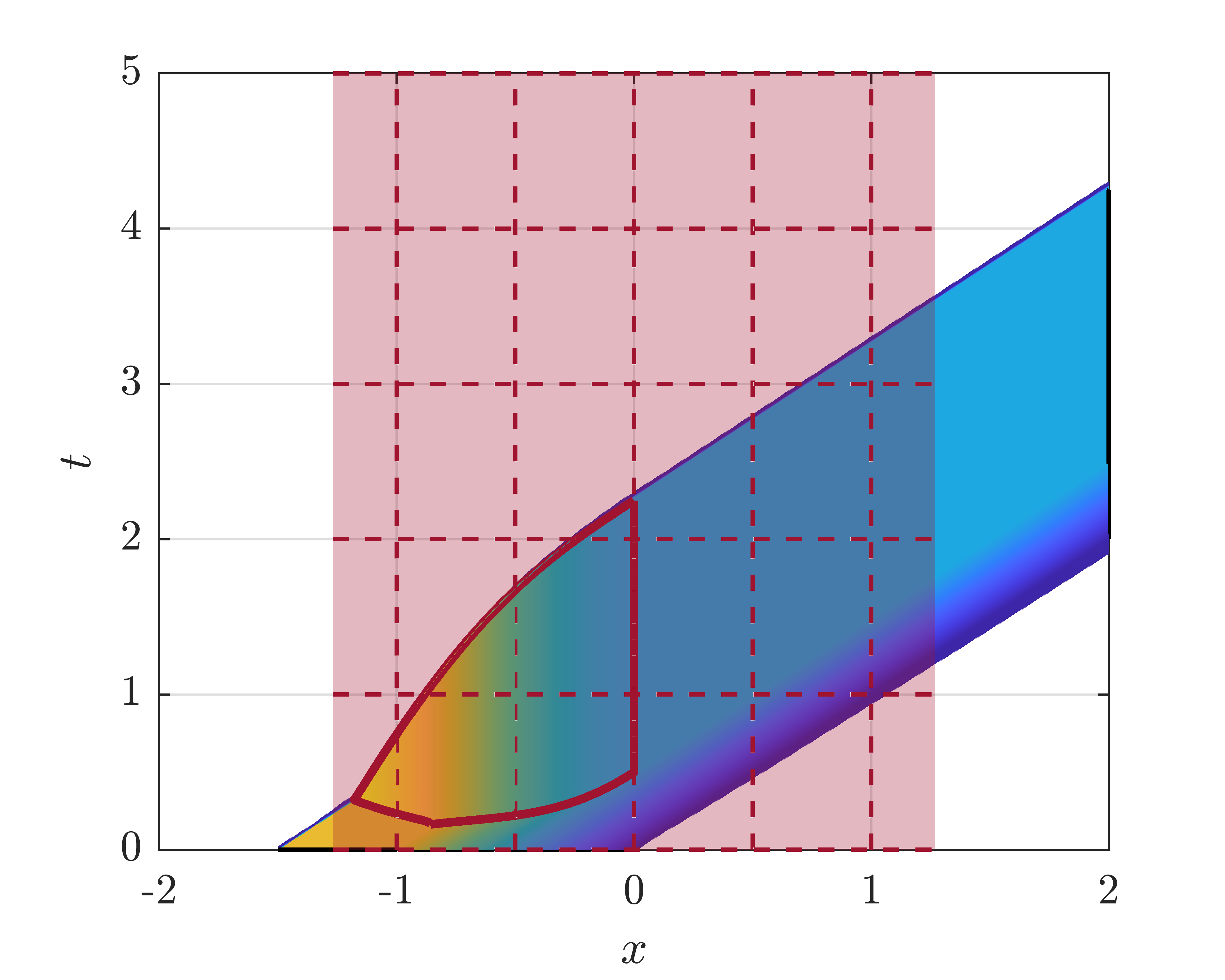}}
    {\includegraphics[scale=0.5,clip,trim=3 0 25 5]{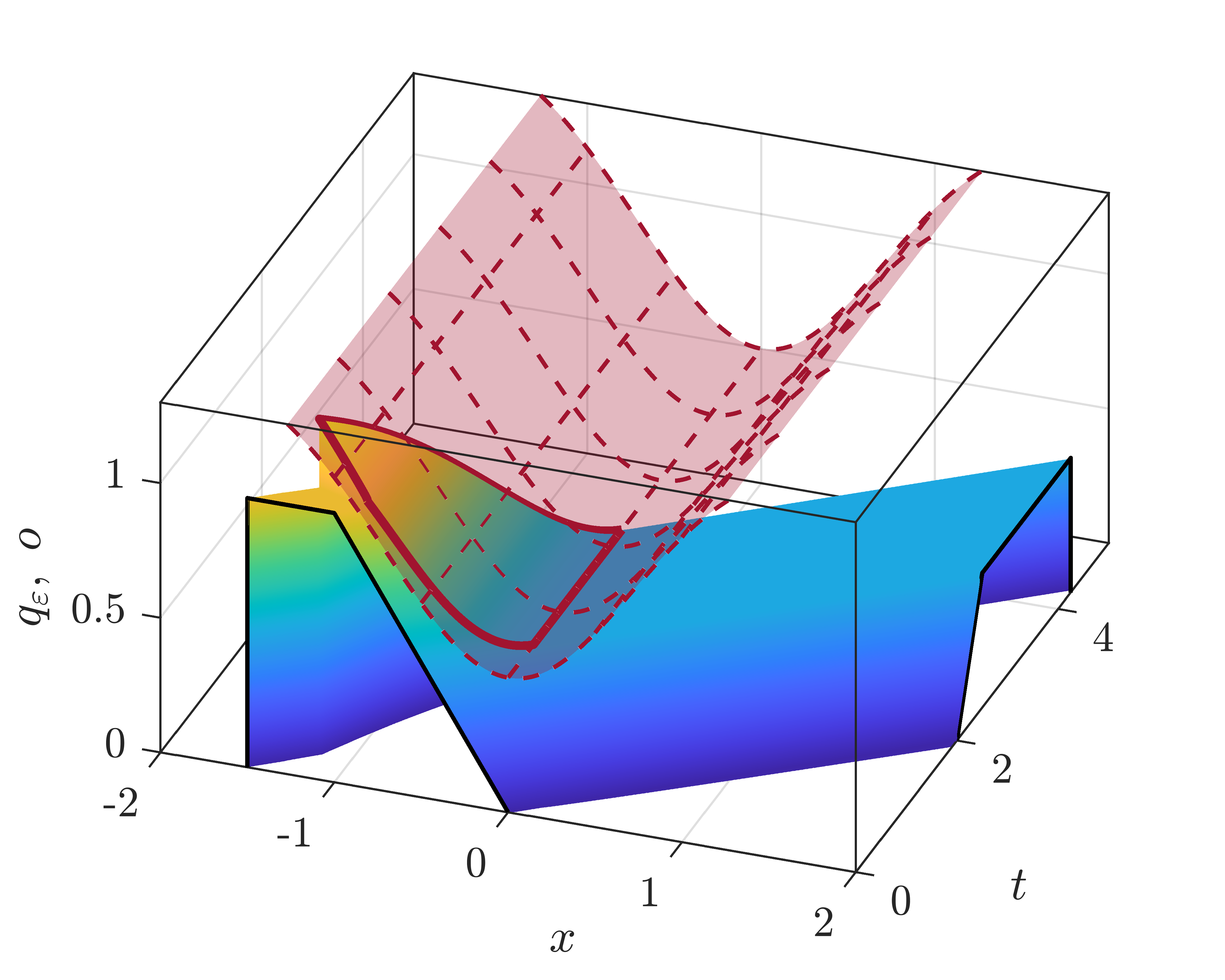}}
    \caption{Plots of $q_\eps$ (color gradient) and \textcolor{matlabred}{$o_1$} (see \eqref{eq:example_defi_obstacle}) from several perspectives. The initial datum is $q_0^1$ (see \eqref{eq:example_defi_initial-datum}). The thick red line surrounds the coincidence region $\lbrace \lim_{\eps \searrow 0} q_\eps = o \rbrace$.}
    \label{3d}
\end{figure}
Choosing the data $q_0^3$ and $o_3$, we obtain the results portrayed in~\cref{q03o3}.
\begin{figure}
    \centering
    \includegraphics[scale=0.5,clip,trim=0 27 12 0 0]{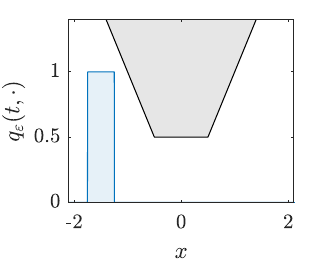}
    \includegraphics[scale=0.5,clip,trim=32 27 12 0 0]{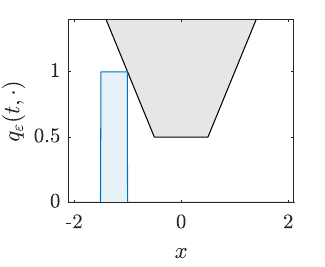}
    \includegraphics[scale=0.5,clip,trim=32 27 12 0 0]{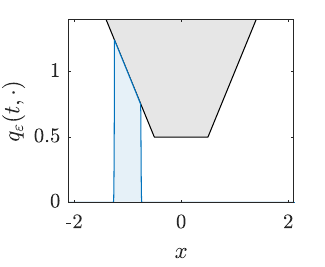}
    \includegraphics[scale=0.5,clip,trim=32 27 12 0 0]{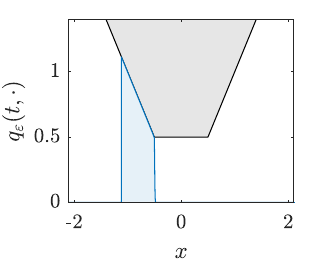}

        \includegraphics[scale=0.5,clip,trim=0 0 12 5]{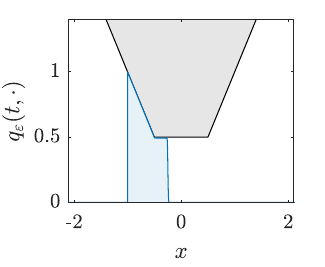}
    \includegraphics[scale=0.5,clip,trim=32 0 12 5]{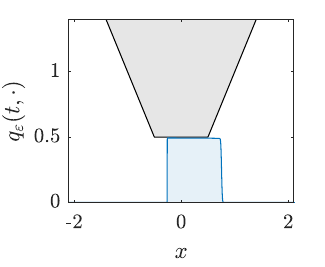}
    \includegraphics[scale=0.5,clip,trim=32 0 12 5]{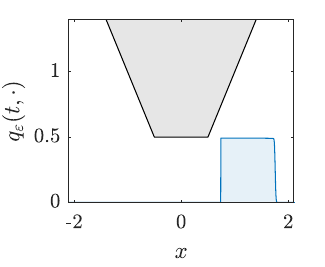}
    \includegraphics[scale=0.5,clip,trim=32 0 12 5]{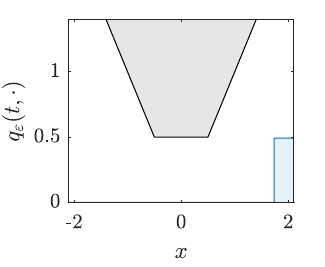}
    \caption{Plots of $o_3$ and \textcolor{matlabblue}{$q_\eps$} for $\eps = 2^{-10}$ and initial datum $q_{0}^{(3)}$ specified in \cref{eq:example_defi_obstacle,eq:example_defi_initial-datum}. Times are $t=0$, $t=0.25$, $t=0.5$, $t=0.75$ (top) and $t=1$, $t=2$, $t=3$, $t=4$ (bottom) from left to right}
    \label{q03o3}
\end{figure}

In fact, the lack of regularity of this data does not prevent the solution from behaving in a physical/intuitive manner. However, at $t=2$ and $t=3$, some minor diffusion effects can be observed in the sense that the solution seems to ``fade out'' to its right. This is a numerical problem and would be far greater if one employs a Lax--Friedrichs scheme~\cite[p.16]{Zhang2003} instead of a Godunov scheme.

The obstacle in the case of $o_2$ and $q_0^1$ (see~\cref{q01o2}) is a little more delicate. Still, the solution behaves in the desired manner. Once the first stalactite of the obstacle is passed, the density between the two stalactites only starts to increase once the second one is hit. After some time, all the ``room'' is used by the solution.
\begin{figure}
    \centering
    \includegraphics[scale=0.5,clip,trim=0 27 12 0 0]{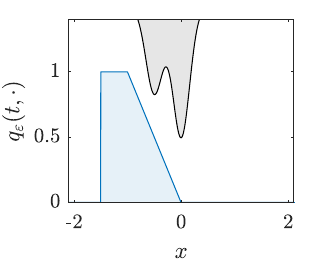}
    \includegraphics[scale=0.5,clip,trim=32 27 12 0 0]{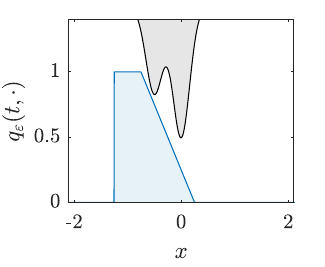}
    \includegraphics[scale=0.5,clip,trim=32 27 12 0 0]{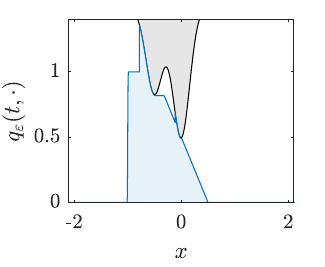}
    \includegraphics[scale=0.5,clip,trim=32 27 12 0 0]{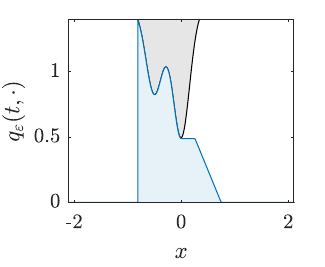}

        \includegraphics[scale=0.5,clip,trim=0 0 12 5]{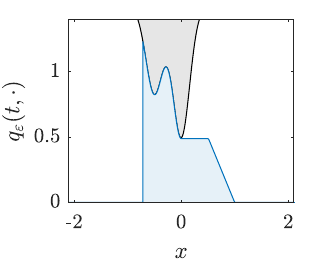}
    \includegraphics[scale=0.5,clip,trim=32 0 12 5]{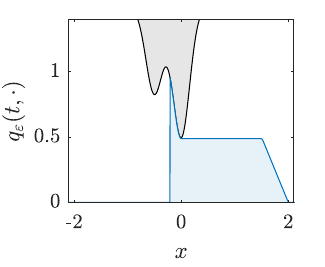}
    \includegraphics[scale=0.5,clip,trim=32 0 12 5]{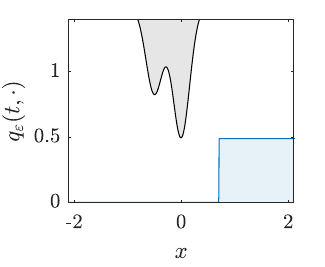}
    \includegraphics[scale=0.5,clip,trim=32 0 12 5]{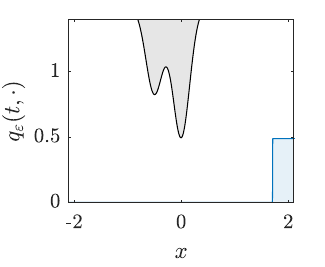}
    \caption{Plots of $o_2$ and \textcolor{matlabblue}{$q_\eps$} for $\eps = 2^{-10}$ and initial datum $q_{0}^{(1)}$ specified in \cref{eq:example_defi_obstacle,eq:example_defi_initial-datum}. Times are $t=0$, $t=0.25$, $t=0.5$, $t=0.75$ (top) and $t=1$, $t=2$, $t=3$, $t=4$ (bottom) from left to right}
    \label{q01o2}
\end{figure}

\subsection{\texorpdfstring{$q_\eps$}{q} as \texorpdfstring{$\eps \searrow 0$}{eps→0}}
To observe the behavior of $q_\eps$ with $\eps$ becoming very small, we simulate~\cref{eq:conservation_law_smooth} with initial datum $q_0^1$ and obstacle $o_1$ for several different $\eps$, as shown in~\cref{epskon}.
\begin{figure}
    \centering
    \includegraphics[scale=0.5,clip,trim=0 30 20 0]{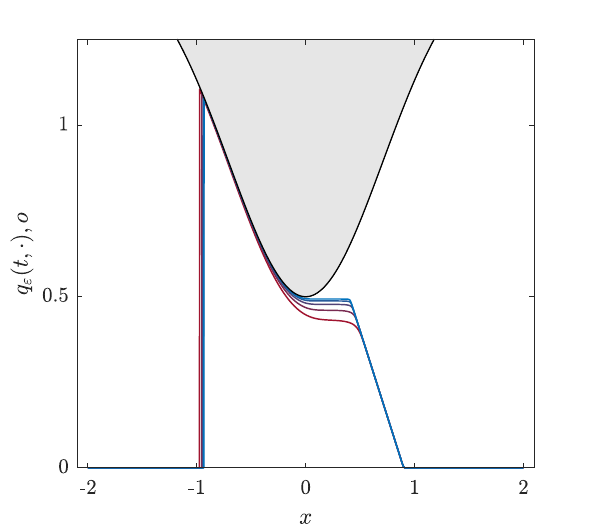}
    \includegraphics[scale=0.5,clip,trim=33 30 20 0]{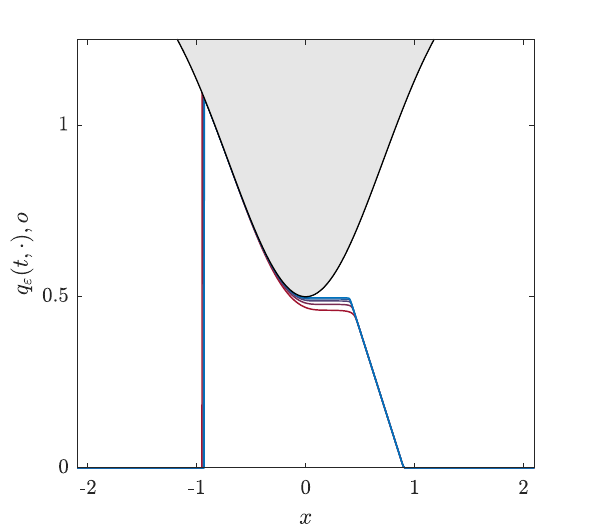}
    \includegraphics[scale=0.5,clip,trim=33 30 20 0]{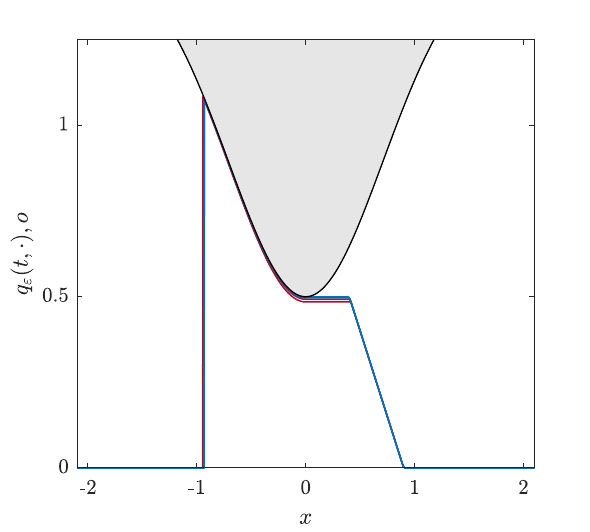}

    \includegraphics[scale=0.5,clip,trim=0 0 20 0]{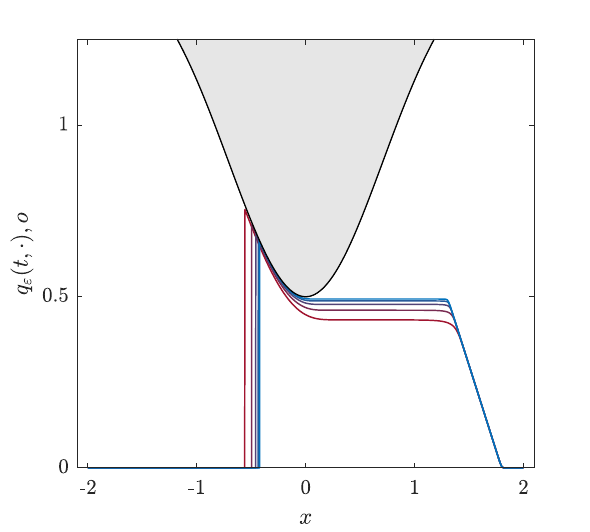}
    \includegraphics[scale=0.5,clip,trim=33 0 20 0]{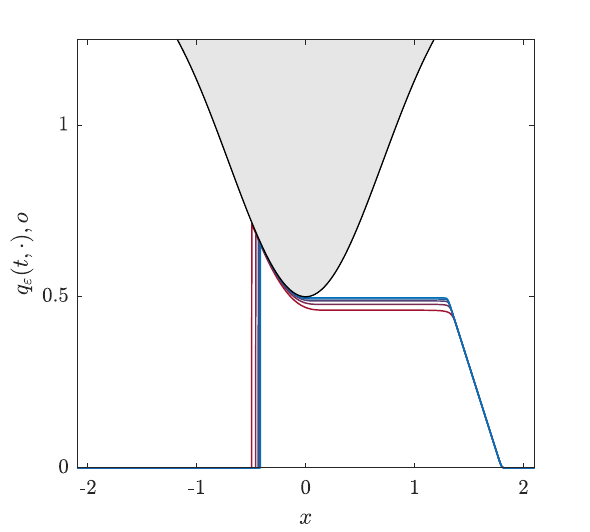}
    \includegraphics[scale=0.5,clip,trim=33 0 20 0]{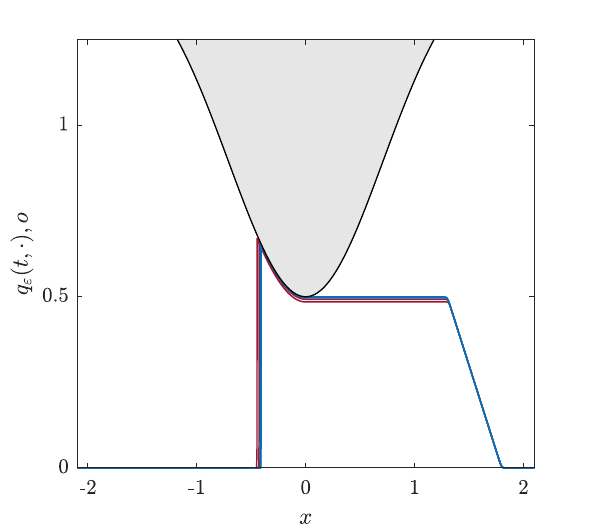}
    \caption{Plots of $o_1$ (see \cref{eq:example_defi_obstacle}) and $q_\eps(t,\cdot)$ for $\eps \in \left\lbrace \textcolor{matlabred}{2^{-6}}, \textcolor{matlabred1}{2^{-7}}, \textcolor{matlabred2}{2^{-8}}, \textcolor{matlabred3}{2^{-9}}, \textcolor{matlabblue}{2^{-10}} \right\rbrace$, where $t=0.9$ in the top row and $t=1.8$ in the bottom row. Here, the initial datum is $q^{(1)}_0$ (see \cref{eq:example_defi_initial-datum}) and $V_\eps$ is as in~\cref{ex:V} (left column), $V_\eps(x) = \tfrac{1}{2}\left(\tanh\big(\tfrac{x}{\eps}\big)+1 \right)$ (middle column) and $V_\eps(x) = \min\left \lbrace\max \left \lbrace \tfrac{x}{\eps}, 0 \right \rbrace, 1 \right \rbrace$ (right column) for all $x \in \mathbb{R}$.}
    \label{epskon}
\end{figure}

We see that for smaller $\eps$, the solution moves closer to the obstacle. This becomes especially evident around the point $x=0$. This behavior can be motivated by the fact that $V_\eps \nearrow 1$ as $\eps \searrow 0$ on $[0, \infty)$ (recall~\cref{asH}). So, the smaller $\eps$, the later $q_\eps$ slows down. \\ 
Thus, as suggested in~\cref{alllim} 1), the convergence of $q_\eps$ as $\eps \searrow 0$ can be identified numerically. 

Moreover, if we choose a different regularization of $V_\eps$, namely, $\mathbb{R} \ni x \mapsto \tfrac{1}{2}\left( \tanh\big(\tfrac{x}{\varepsilon} \big) + 1\right)$ or $\mathbb{R} \ni x \mapsto \min \left \lbrace \max \left \lbrace \tfrac{x}{\eps}, 0 \right \rbrace, 1 \right \rbrace$ (see \cref{epskon}), we see the same behavior. Note that the latter function (e.g., due to non-differentiability) does not meet the requirements in~\cref{asH}.
In fact, the solutions for larger $\eps$ are closer to the solution for $\eps = \tfrac{1}{1024}$ than in \cref{epskon}. This is due to the $\tanh$-regularization being more precise.

\subsection{Visualization of \texorpdfstring{$V_\eps$}{V eps}}
Motivated by~\cref{thm:Limit_V_Rankine-Hugoniot}, we hypothesize that if $t \in [0, T]$ and there exists an interval $I(t)$ such that $\lim_{\eps \searrow 0} q_\eps = o$ on $\lbrace t \rbrace \times I(t)$,  $\lim_{\eps \searrow 0} V_\eps(o(x)-q_\eps(t, x)) = \tfrac{o\left(b_t \right)}{o(x)}$ holds for all $x \in I(t)$, where $b_t \coloneqq \sup I(t)$. 
As we can not compute $\lim_{\eps \searrow 0} q_\eps$, we want to analyze $V_\eps(o-q_\eps)$ for the small value of $\eps = \tfrac{1}{1024}$. In fact, at least numerically, our examples in~\cref{Vkon} show the desired behavior.
\begin{figure}
    \centering
    \includegraphics[scale=0.5,clip,trim=0 27 20 0 0]{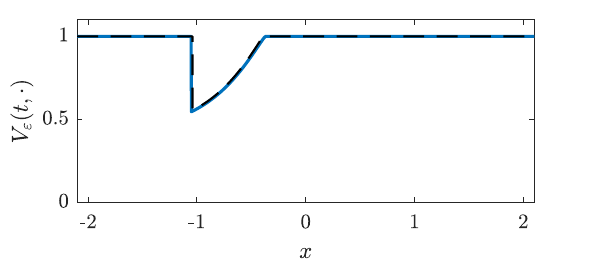}
    \includegraphics[scale=0.5,clip,trim=35 27 0 0 0]{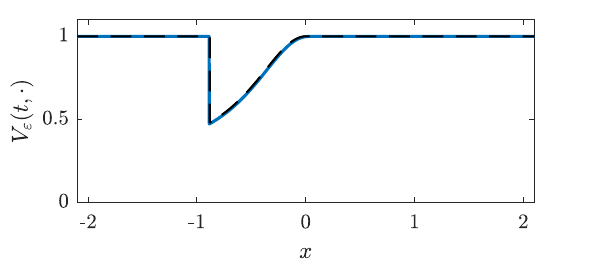}

    \includegraphics[scale=0.5,clip,trim=0 27 20 0 0]{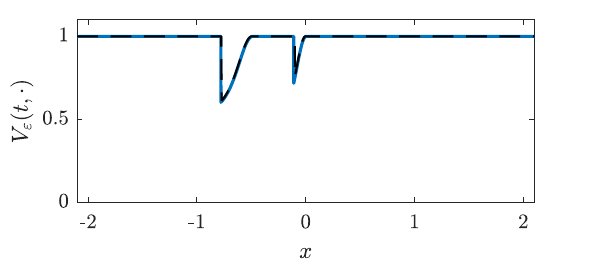}
    \includegraphics[scale=0.5,clip,trim=35 27 0 0 0]{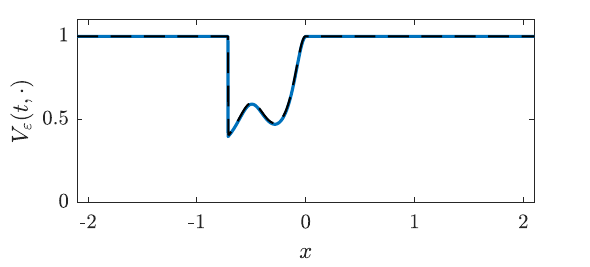}

    \includegraphics[scale=0.5,clip,trim=0 0 20 0 0]{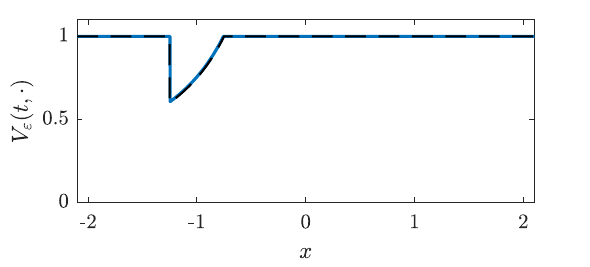}
    \includegraphics[scale=0.5,clip,trim=35 0 0 0 0]{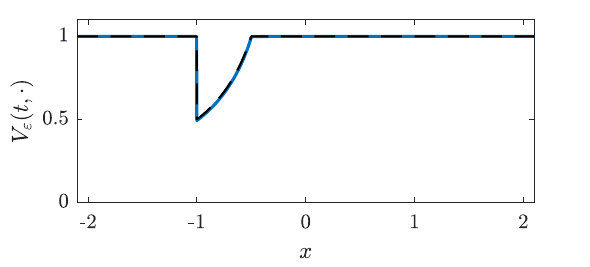}

    \caption{$V^\star$ according to \cref{thm:Limit_V_Rankine-Hugoniot} and \textcolor{matlabblue}{$V_\eps(o-q_\eps)$} with $\eps=2^{-10}$ for times $t=0.25, 1$ with $q_0^1$, $o_1$ (top from left to right), $t=0.5, 1$ with $q^0_2$, $o_2$ (middle from left to right), $t = 0.5, 1$ with $q_0^3$, $o_3$ (bottom from left to right). See \cref{eq:example_defi_obstacle,eq:example_defi_initial-datum} for the definition of the obstacles and initial data, respectively.}
    \label{Vkon}
\end{figure}
\section{Open problems---future research}
The most important question not answered in this work is whether the solution to~\cref{eq:discont_prob} is unique and if it is independent of the regularization $V_\eps$.\\
Moreover, several generalizations are straightforward to some degree:
\begin{itemize}
    \item The same approach works if we do not assume a constant velocity \(1\) if the obstacle is not present but rather a space- and time-dependent velocity satisfying required regularity for well-posedness.
    \item One can also consider nonlinear conservation laws \cite{Bressan2000}, i.e.,\ 
    \[
\partial_{t}q +\partial_{x}f(q)=0
    \]
    with the same scaling argument, resulting in
    \[
\partial_{t}q +\partial_{x}\big(V_{\eps}(o-q)f(q)\big)=0
    \]
    and \(V_{\eps}\) the regularization of the Heaviside function as in \cref{asH}.
    \item The same argument can be made for nonlocal conservation laws \cite{keimernonlocalbalance2017}.
    \item In the case of systems of conservation laws \cite{Bressan2000} and multi-d conservation laws, the presented approach may require some refinement. Particularly in the multi-d case, researchers may choose to steer toward the direction of the velocity field, which cannot be represented by a simple rescaling.
\end{itemize}
\section*{Data Availability}
There is no associated data.

\section*{Acknowledgments}
L.~Pflug and J.~Rodestock have been supported by the DFG -- Project-ID 416229255 -- SFB 1411.

\bibliography{obstacle}


\begin{thebibliography}{54}
\ifx \bisbn   \undefined \def \bisbn  #1{ISBN #1}\fi
\ifx \binits  \undefined \def \binits#1{#1}\fi
\ifx \bauthor  \undefined \def \bauthor#1{#1}\fi
\ifx \batitle  \undefined \def \batitle#1{#1}\fi
\ifx \bjtitle  \undefined \def \bjtitle#1{#1}\fi
\ifx \bvolume  \undefined \def \bvolume#1{\textbf{#1}}\fi
\ifx \byear  \undefined \def \byear#1{#1}\fi
\ifx \bissue  \undefined \def \bissue#1{#1}\fi
\ifx \bfpage  \undefined \def \bfpage#1{#1}\fi
\ifx \blpage  \undefined \def \blpage #1{#1}\fi
\ifx \burl  \undefined \def \burl#1{\textsf{#1}}\fi
\ifx \doiurl  \undefined \def \doiurl#1{\url{https://doi.org/#1}}\fi
\ifx \betal  \undefined \def \betal{\textit{et al.}}\fi
\ifx \binstitute  \undefined \def \binstitute#1{#1}\fi
\ifx \binstitutionaled  \undefined \def \binstitutionaled#1{#1}\fi
\ifx \bctitle  \undefined \def \bctitle#1{#1}\fi
\ifx \beditor  \undefined \def \beditor#1{#1}\fi
\ifx \bpublisher  \undefined \def \bpublisher#1{#1}\fi
\ifx \bbtitle  \undefined \def \bbtitle#1{#1}\fi
\ifx \bedition  \undefined \def \bedition#1{#1}\fi
\ifx \bseriesno  \undefined \def \bseriesno#1{#1}\fi
\ifx \blocation  \undefined \def \blocation#1{#1}\fi
\ifx \bsertitle  \undefined \def \bsertitle#1{#1}\fi
\ifx \bsnm \undefined \def \bsnm#1{#1}\fi
\ifx \bsuffix \undefined \def \bsuffix#1{#1}\fi
\ifx \bparticle \undefined \def \bparticle#1{#1}\fi
\ifx \barticle \undefined \def \barticle#1{#1}\fi
\bibcommenthead
\ifx \bconfdate \undefined \def \bconfdate #1{#1}\fi
\ifx \botherref \undefined \def \botherref #1{#1}\fi
\ifx \url \undefined \def \url#1{\textsf{#1}}\fi
\ifx \bchapter \undefined \def \bchapter#1{#1}\fi
\ifx \bbook \undefined \def \bbook#1{#1}\fi
\ifx \bcomment \undefined \def \bcomment#1{#1}\fi
\ifx \oauthor \undefined \def \oauthor#1{#1}\fi
\ifx \citeauthoryear \undefined \def \citeauthoryear#1{#1}\fi
\ifx \endbibitem  \undefined \def \endbibitem {}\fi
\ifx \bconflocation  \undefined \def \bconflocation#1{#1}\fi
\ifx \arxivurl  \undefined \def \arxivurl#1{\textsf{#1}}\fi
\csname PreBibitemsHook\endcsname

\bibitem[\protect\citeauthoryear{Berthelin et~al.}{2007}]{1berthelin2008model}
\begin{barticle}
\bauthor{\bsnm{Berthelin}, \binits{F.}},
\bauthor{\bsnm{Degond}, \binits{P.}},
\bauthor{\bsnm{Delitala}, \binits{M.}},
\bauthor{\bsnm{Rascle}, \binits{M.}}:
\batitle{A model for the formation and evolution of traffic jams}.
\bjtitle{Archive for Rational Mechanics and Analysis}
\bvolume{187}(\bissue{2}),
\bfpage{185}--\blpage{220}
(\byear{2007})
\doiurl{10.1007/s00205-007-0061-9}
\end{barticle}
\endbibitem

\bibitem[\protect\citeauthoryear{Andreianov et~al.}{2016a}]{1andreianov}
\begin{barticle}
\bauthor{\bsnm{Andreianov}, \binits{B.}},
\bauthor{\bsnm{Donadello}, \binits{C.}},
\bauthor{\bsnm{Razafison}, \binits{U.}},
\bauthor{\bsnm{Rosini}, \binits{M.D.}}:
\batitle{Qualitative behaviour and numerical approximation of solutions to
  conservation laws with non-local point constraints on the flux and modeling
  of crowd dynamics at the bottlenecks}.
\bjtitle{ESAIM Math. Model. Numer. Anal.}
\bvolume{50}(\bissue{5}),
\bfpage{1269}--\blpage{1287}
(\byear{2016})
\doiurl{10.1051/m2an/2015078}
\end{barticle}
\endbibitem

\bibitem[\protect\citeauthoryear{Andreianov et~al.}{2016b}]{1andreianov2}
\begin{barticle}
\bauthor{\bsnm{Andreianov}, \binits{B.}},
\bauthor{\bsnm{Donadello}, \binits{C.}},
\bauthor{\bsnm{Rosini}, \binits{M.D.}}:
\batitle{A second-order model for vehicular traffics with local point
  constraints on the flow}.
\bjtitle{Math. Models Methods Appl. Sci.}
\bvolume{26}(\bissue{4}),
\bfpage{751}--\blpage{802}
(\byear{2016})
\doiurl{10.1142/S0218202516500172}
\end{barticle}
\endbibitem

\bibitem[\protect\citeauthoryear{Chalons et~al.}{2013}]{1chalons}
\begin{barticle}
\bauthor{\bsnm{Chalons}, \binits{C.}},
\bauthor{\bsnm{Goatin}, \binits{P.}},
\bauthor{\bsnm{Seguin}, \binits{N.}}:
\batitle{General constrained conservation laws. {A}pplication to pedestrian
  flow modeling}.
\bjtitle{Netw. Heterog. Media}
\bvolume{8}(\bissue{2}),
\bfpage{433}--\blpage{463}
(\byear{2013})
\doiurl{10.3934/nhm.2013.8.433}
\end{barticle}
\endbibitem

\bibitem[\protect\citeauthoryear{Berthelin}{2002}]{1berthelin}
\begin{barticle}
\bauthor{\bsnm{Berthelin}, \binits{F.}}:
\batitle{Existence and weak stability for a pressureless model with unilateral
  constraint}.
\bjtitle{Math. Models Methods Appl. Sci.}
\bvolume{12}(\bissue{2}),
\bfpage{249}--\blpage{272}
(\byear{2002})
\doiurl{10.1142/S0218202502001635}
\end{barticle}
\endbibitem

\bibitem[\protect\citeauthoryear{De~Nitti et~al.}{2024}]{denitti2024pointwise}
\begin{botherref}
\oauthor{\bsnm{De~Nitti}, \binits{N.}},
\oauthor{\bsnm{Serre}, \binits{D.}},
\oauthor{\bsnm{Zuazua}, \binits{E.}}:
Pointwise constraints for scalar conservation laws with positive wave velocity.
cvgmt preprint
(2024).
\url{http://cvgmt.sns.it/paper/6472/}
\end{botherref}
\endbibitem

\bibitem[\protect\citeauthoryear{Colombo and Goatin}{2007}]{1colombo}
\begin{barticle}
\bauthor{\bsnm{Colombo}, \binits{R.M.}},
\bauthor{\bsnm{Goatin}, \binits{P.}}:
\batitle{A well posed conservation law with a variable unilateral constraint}.
\bjtitle{J. Differential Equations}
\bvolume{234}(\bissue{2}),
\bfpage{654}--\blpage{675}
(\byear{2007})
\doiurl{10.1016/j.jde.2006.10.014}
\end{barticle}
\endbibitem

\bibitem[\protect\citeauthoryear{Dymski et~al.}{2018}]{1dymski}
\begin{barticle}
\bauthor{\bsnm{Dymski}, \binits{N.S.}},
\bauthor{\bsnm{Goatin}, \binits{P.}},
\bauthor{\bsnm{Rosini}, \binits{M.D.}}:
\batitle{Existence of \(\boldsymbol{BV}\) solutions for a non-conservative
  constrained {A}w-{R}ascle-{Z}hang model for vehicular traffic}.
\bjtitle{J. Math. Anal. Appl.}
\bvolume{467}(\bissue{1}),
\bfpage{45}--\blpage{66}
(\byear{2018})
\doiurl{10.1016/j.jmaa.2018.07.025}
\end{barticle}
\endbibitem

\bibitem[\protect\citeauthoryear{Garavello and Goatin}{2011}]{1garavello}
\begin{barticle}
\bauthor{\bsnm{Garavello}, \binits{M.}},
\bauthor{\bsnm{Goatin}, \binits{P.}}:
\batitle{The {A}w-{R}ascle traffic model with locally constrained flow}.
\bjtitle{J. Math. Anal. Appl.}
\bvolume{378}(\bissue{2}),
\bfpage{634}--\blpage{648}
(\byear{2011})
\doiurl{10.1016/j.jmaa.2011.01.033}
\end{barticle}
\endbibitem

\bibitem[\protect\citeauthoryear{Garavello and Villa}{2017}]{1garavello2}
\begin{barticle}
\bauthor{\bsnm{Garavello}, \binits{M.}},
\bauthor{\bsnm{Villa}, \binits{S.}}:
\batitle{The {C}auchy problem for the {A}w-{R}ascle-{Z}hang traffic model with
  locally constrained flow}.
\bjtitle{J. Hyperbolic Differ. Equ.}
\bvolume{14}(\bissue{3}),
\bfpage{393}--\blpage{414}
(\byear{2017})
\doiurl{10.1142/S0219891617500138}
\end{barticle}
\endbibitem

\bibitem[\protect\citeauthoryear{Bayen et~al.}{2022}]{Bayen2022}
\begin{barticle}
\bauthor{\bsnm{Bayen}, \binits{A.}},
\bauthor{\bsnm{Friedrich}, \binits{J.}},
\bauthor{\bsnm{Keimer}, \binits{A.}},
\bauthor{\bsnm{Pflug}, \binits{L.}},
\bauthor{\bsnm{Veeravalli}, \binits{T.}}:
\batitle{Modeling multilane traffic with moving obstacles by nonlocal balance
  laws}.
\bjtitle{SIAM Journal on Applied Dynamical Systems}
\bvolume{21}(\bissue{2}),
\bfpage{1495}--\blpage{1538}
(\byear{2022})
\doiurl{10.1137/20m1366654}
\end{barticle}
\endbibitem

\bibitem[\protect\citeauthoryear{Levi}{2001}]{1levi}
\begin{barticle}
\bauthor{\bsnm{Levi}, \binits{L.}}:
\batitle{Obstacle problems for scalar conservation laws}.
\bjtitle{M2AN Math. Model. Numer. Anal.}
\bvolume{35}(\bissue{3}),
\bfpage{575}--\blpage{593}
(\byear{2001})
\doiurl{10.1051/m2an:2001127}
\end{barticle}
\endbibitem

\bibitem[\protect\citeauthoryear{Rodrigues}{2004}]{1rodrigues}
\begin{barticle}
\bauthor{\bsnm{Rodrigues}, \binits{J.F.}}:
\batitle{On hyperbolic variational inequalities of first order and some
  applications}.
\bjtitle{Monatsh. Math.}
\bvolume{142}(\bissue{1-2}),
\bfpage{157}--\blpage{177}
(\byear{2004})
\doiurl{10.1007/s00605-004-0238-3}
\end{barticle}
\endbibitem

\bibitem[\protect\citeauthoryear{Rodrigues}{2002}]{1rodrigues2}
\begin{botherref}
\oauthor{\bsnm{Rodrigues}, \binits{J.F.}}:
On the hyperbolic obstacle problem of first order.
vol. 23,
pp. 253--266
(2002).
\doiurl{10.1142/S0252959902000249} .
Dedicated to the memory of Jacques-Louis Lions.
\url{https://doi.org/10.1142/S0252959902000249}
\end{botherref}
\endbibitem

\bibitem[\protect\citeauthoryear{Goudon and Vasseur}{2016}]{1goudon}
\begin{barticle}
\bauthor{\bsnm{Goudon}, \binits{T.}},
\bauthor{\bsnm{Vasseur}, \binits{A.}}:
\batitle{On a model for mixture flows: derivation, dissipation and stability
  properties}.
\bjtitle{Arch. Ration. Mech. Anal.}
\bvolume{220}(\bissue{1}),
\bfpage{1}--\blpage{35}
(\byear{2016})
\doiurl{10.1007/s00205-015-0925-3}
\end{barticle}
\endbibitem

\bibitem[\protect\citeauthoryear{Saldanha~da Gama et~al.}{2012}]{1saldanha}
\begin{barticle}
\bauthor{\bsnm{Gama}, \binits{R.M.}},
\bauthor{\bsnm{Pedrosa~Filho}, \binits{J.J.}},
\bauthor{\bsnm{Martins-Costa}, \binits{M.L.}}:
\batitle{Modeling the saturation process of flows through rigid porous media by
  the solution of a nonlinear hyperbolic system with one constrained unknown}.
\bjtitle{ZAMM Z. Angew. Math. Mech.}
\bvolume{92}(\bissue{11-12}),
\bfpage{921}--\blpage{936}
(\byear{2012})
\doiurl{10.1002/zamm.201100031}
\end{barticle}
\endbibitem

\bibitem[\protect\citeauthoryear{Berthelin and
  Bouchut}{2003}]{1berthelin-bouchut}
\begin{barticle}
\bauthor{\bsnm{Berthelin}, \binits{F.}},
\bauthor{\bsnm{Bouchut}, \binits{F.}}:
\batitle{Weak solutions for a hyperbolic system with unilateral constraint and
  mass loss}.
\bjtitle{Ann. Inst. H. Poincar\'{e} C Anal. Non Lin\'{e}aire}
\bvolume{20}(\bissue{6}),
\bfpage{975}--\blpage{997}
(\byear{2003})
\doiurl{10.1016/S0294-1449(03)00012-X}
\end{barticle}
\endbibitem

\bibitem[\protect\citeauthoryear{Rossi et~al.}{2020}]{Rossi2020}
\begin{barticle}
\bauthor{\bsnm{Rossi}, \binits{E.}},
\bauthor{\bsnm{Weißen}, \binits{J.}},
\bauthor{\bsnm{Goatin}, \binits{P.}},
\bauthor{\bsnm{G\"{o}ttlich}, \binits{S.}}:
\batitle{Well-posedness of a non-local model for material flow on conveyor
  belts}.
\bjtitle{ESAIM: Mathematical Modelling and Numerical Analysis}
\bvolume{54}(\bissue{2}),
\bfpage{679}--\blpage{704}
(\byear{2020})
\doiurl{10.1051/m2an/2019062}
\end{barticle}
\endbibitem

\bibitem[\protect\citeauthoryear{Fernández-Real and
  Figalli}{2020}]{FernandezReal2020}
\begin{barticle}
\bauthor{\bsnm{Fernández-Real}, \binits{X.}},
\bauthor{\bsnm{Figalli}, \binits{A.}}:
\batitle{On the obstacle problem for the 1d wave equation}.
\bjtitle{Mathematics in Engineering}
\bvolume{2}(\bissue{4}),
\bfpage{584}--\blpage{597}
(\byear{2020})
\doiurl{10.3934/mine.2020026}
\end{barticle}
\endbibitem

\bibitem[\protect\citeauthoryear{Bellomo et~al.}{2022}]{1bellomo2022towards}
\begin{barticle}
\bauthor{\bsnm{Bellomo}, \binits{N.}},
\bauthor{\bsnm{Gibelli}, \binits{L.}},
\bauthor{\bsnm{Quaini}, \binits{A.}},
\bauthor{\bsnm{Reali}, \binits{A.}}:
\batitle{Towards a mathematical theory of behavioral human crowds}.
\bjtitle{Mathematical Models and Methods in Applied Sciences}
\bvolume{32}(\bissue{02}),
\bfpage{321}--\blpage{358}
(\byear{2022})
\end{barticle}
\endbibitem

\bibitem[\protect\citeauthoryear{Lions and Stampacchia}{1967}]{2lions}
\begin{barticle}
\bauthor{\bsnm{Lions}, \binits{J.-L.}},
\bauthor{\bsnm{Stampacchia}, \binits{G.}}:
\batitle{Variational inequalities}.
\bjtitle{Comm. Pure Appl. Math.}
\bvolume{20},
\bfpage{493}--\blpage{519}
(\byear{1967})
\doiurl{10.1002/cpa.3160200302}
\end{barticle}
\endbibitem

\bibitem[\protect\citeauthoryear{Kinderlehrer and Stampacchia}{1980}]{2kinder}
\begin{bbook}
\bauthor{\bsnm{Kinderlehrer}, \binits{D.}},
\bauthor{\bsnm{Stampacchia}, \binits{G.}}:
\bbtitle{An Introduction to Variational Inequalities and Their Applications}.
\bsertitle{Pure and Applied Mathematics},
vol. \bseriesno{88},
p. \bfpage{313}.
\bpublisher{Academic Press, Inc. [Harcourt Brace Jovanovich, Publishers]},
\blocation{New York-London}
(\byear{1980})
\end{bbook}
\endbibitem

\bibitem[\protect\citeauthoryear{Rodrigues}{1987}]{2rodriguesbook}
\begin{bbook}
\bauthor{\bsnm{Rodrigues}, \binits{J.-F.}}:
\bbtitle{Obstacle Problems in Mathematical Physics}.
\bsertitle{North-Holland Mathematics Studies},
vol. \bseriesno{134},
p. \bfpage{352}.
\bpublisher{North-Holland Publishing Co.},
\blocation{Amsterdam}
(\byear{1987}).
\bcomment{Notas de Matem\'{a}tica, 114. [Mathematical Notes]}
\end{bbook}
\endbibitem

\bibitem[\protect\citeauthoryear{Mignot and Puel}{1976}]{2mignot}
\begin{barticle}
\bauthor{\bsnm{Mignot}, \binits{F.}},
\bauthor{\bsnm{Puel}, \binits{J.-P.}}:
\batitle{In\'{e}quations variationnelles et quasivariationnelles hyperboliques
  du premier ordre}.
\bjtitle{J. Math. Pures Appl. (9)}
\bvolume{55}(\bissue{3}),
\bfpage{353}--\blpage{378}
(\byear{1976})
\end{barticle}
\endbibitem

\bibitem[\protect\citeauthoryear{Br\'{e}zis}{1972}]{2brezis}
\begin{barticle}
\bauthor{\bsnm{Br\'{e}zis}, \binits{H.}}:
\batitle{Probl\`emes unilat\'{e}raux}.
\bjtitle{J. Math. Pures Appl. (9)}
\bvolume{51},
\bfpage{1}--\blpage{168}
(\byear{1972})
\end{barticle}
\endbibitem

\bibitem[\protect\citeauthoryear{Rudd and Schmitt}{2002}]{2rudd}
\begin{barticle}
\bauthor{\bsnm{Rudd}, \binits{M.}},
\bauthor{\bsnm{Schmitt}, \binits{K.}}:
\batitle{Variational inequalities of elliptic and parabolic type}.
\bjtitle{Taiwanese J. Math.}
\bvolume{6}(\bissue{3}),
\bfpage{287}--\blpage{322}
(\byear{2002})
\doiurl{10.11650/twjm/1500558298}
\end{barticle}
\endbibitem

\bibitem[\protect\citeauthoryear{Korte et~al.}{2009}]{2korte}
\begin{barticle}
\bauthor{\bsnm{Korte}, \binits{R.}},
\bauthor{\bsnm{Kuusi}, \binits{T.}},
\bauthor{\bsnm{Siljander}, \binits{J.}}:
\batitle{Obstacle problem for nonlinear parabolic equations}.
\bjtitle{J. Differential Equations}
\bvolume{246}(\bissue{9}),
\bfpage{3668}--\blpage{3680}
(\byear{2009})
\doiurl{10.1016/j.jde.2009.02.006}
\end{barticle}
\endbibitem

\bibitem[\protect\citeauthoryear{Rodrigues and Santos}{2012}]{2rodrigues}
\begin{barticle}
\bauthor{\bsnm{Rodrigues}, \binits{J.F.}},
\bauthor{\bsnm{Santos}, \binits{L.}}:
\batitle{Quasivariational solutions for first order quasilinear equations with
  gradient constraint}.
\bjtitle{Arch. Ration. Mech. Anal.}
\bvolume{205}(\bissue{2}),
\bfpage{493}--\blpage{514}
(\byear{2012})
\doiurl{10.1007/s00205-012-0511-x}
\end{barticle}
\endbibitem

\bibitem[\protect\citeauthoryear{Chalub and Rodrigues}{2006}]{chalub}
\begin{barticle}
\bauthor{\bsnm{Chalub}, \binits{F.A.C.C.}},
\bauthor{\bsnm{Rodrigues}, \binits{J.F.}}:
\batitle{A class of kinetic models for chemotaxis with threshold to prevent
  overcrowding}.
\bjtitle{Port. Math. (N.S.)}
\bvolume{63}(\bissue{2}),
\bfpage{227}--\blpage{250}
(\byear{2006})
\end{barticle}
\endbibitem

\bibitem[\protect\citeauthoryear{B\"{o}gelein et~al.}{2011}]{2bogelein}
\begin{barticle}
\bauthor{\bsnm{B\"{o}gelein}, \binits{V.}},
\bauthor{\bsnm{Duzaar}, \binits{F.}},
\bauthor{\bsnm{Mingione}, \binits{G.}}:
\batitle{Degenerate problems with irregular obstacles}.
\bjtitle{J. Reine Angew. Math.}
\bvolume{650},
\bfpage{107}--\blpage{160}
(\byear{2011})
\doiurl{10.1515/CRELLE.2011.006}
\end{barticle}
\endbibitem

\bibitem[\protect\citeauthoryear{Amorim et~al.}{2017}]{1amorim}
\begin{barticle}
\bauthor{\bsnm{Amorim}, \binits{P.}},
\bauthor{\bsnm{Neves}, \binits{W.}},
\bauthor{\bsnm{Rodrigues}, \binits{J.F.}}:
\batitle{The obstacle-mass constraint problem for hyperbolic conservation laws.
  {S}olvability}.
\bjtitle{Ann. Inst. H. Poincar\'{e} C Anal. Non Lin\'{e}aire}
\bvolume{34}(\bissue{1}),
\bfpage{221}--\blpage{248}
(\byear{2017})
\doiurl{10.1016/j.anihpc.2015.11.003}
\end{barticle}
\endbibitem

\bibitem[\protect\citeauthoryear{Kružkov}{1970}]{Kruzkov1970}
\begin{barticle}
\bauthor{\bsnm{Kružkov}, \binits{S.N.}}:
\batitle{First order quasilinear equations in several independent variables}.
\bjtitle{Mathematics of the USSR-Sbornik}
\bvolume{10}(\bissue{2}),
\bfpage{217}
(\byear{1970})
\doiurl{10.1070/SM1970v010n02ABEH002156}
\end{barticle}
\endbibitem

\bibitem[\protect\citeauthoryear{Alt}{2016}]{alt2016eng}
\begin{bbook}
\bauthor{\bsnm{Alt}, \binits{H.W.}}:
\bbtitle{Linear Functional Analysis}.
\bpublisher{Springer},
\blocation{London}
(\byear{2016}).
\doiurl{10.1007/978-1-4471-7280-2} .
\burl{http://dx.doi.org/10.1007/978-1-4471-7280-2}
\end{bbook}
\endbibitem

\bibitem[\protect\citeauthoryear{Evans and Gariepy}{2015}]{evans}
\begin{bbook}
\bauthor{\bsnm{Evans}, \binits{L.C.}},
\bauthor{\bsnm{Gariepy}, \binits{R.F.}}:
\bbtitle{Measure Theory and Fine Properties of Functions, Revised Edition}.
\bpublisher{Chapman and Hall/CRC},
\blocation{Boca Raton}
(\byear{2015}).
\doiurl{10.1201/b18333} .
\burl{http://dx.doi.org/10.1201/b18333}
\end{bbook}
\endbibitem

\bibitem[\protect\citeauthoryear{Godlewski and Raviart}{1991}]{godlewski1991}
\begin{bbook}
\bauthor{\bsnm{Godlewski}, \binits{E.}},
\bauthor{\bsnm{Raviart}, \binits{P.A.}}:
\bbtitle{Hyperbolic Systems of Conservation Laws}.
\bsertitle{Math{\'e}matiques \& applications}.
\bpublisher{Ellipses},
\blocation{Paris}
(\byear{1991}).
\burl{https://books.google.de/books?id=X3qyvAEACAAJ}
\end{bbook}
\endbibitem

\bibitem[\protect\citeauthoryear{Brezis}{2010}]{brezis2010functional}
\begin{bbook}
\bauthor{\bsnm{Brezis}, \binits{H.}}:
\bbtitle{Functional Analysis, Sobolev Spaces and Partial Differential
  Equations}.
\bpublisher{Springer},
\blocation{New York}
(\byear{2010}).
\doiurl{10.1007/978-0-387-70914-7} .
\burl{http://dx.doi.org/10.1007/978-0-387-70914-7}
\end{bbook}
\endbibitem

\bibitem[\protect\citeauthoryear{Simon}{1986}]{Simon1986CompactSI}
\begin{barticle}
\bauthor{\bsnm{Simon}, \binits{J.}}:
\batitle{Compact sets in the space \({L}^{p}(0,{T}; {B})\)}.
\bjtitle{Annali di Matematica Pura ed Applicata}
\bvolume{146},
\bfpage{65}--\blpage{96}
(\byear{1986})
\end{barticle}
\endbibitem

\bibitem[\protect\citeauthoryear{Milgrom and Segal}{2002}]{Milgrom2002}
\begin{barticle}
\bauthor{\bsnm{Milgrom}, \binits{P.}},
\bauthor{\bsnm{Segal}, \binits{I.}}:
\batitle{Envelope theorems for arbitrary choice sets}.
\bjtitle{Econometrica}
\bvolume{70}(\bissue{2}),
\bfpage{583}--\blpage{601}
(\byear{2002})
\doiurl{10.1111/1468-0262.00296}
\end{barticle}
\endbibitem

\bibitem[\protect\citeauthoryear{Nocedal and
  Wright}{2006}]{nocedal2006numerical}
\begin{bbook}
\bauthor{\bsnm{Nocedal}, \binits{J.}},
\bauthor{\bsnm{Wright}, \binits{S.}}:
\bbtitle{Numerical Optimization}.
\bpublisher{Springer},
\blocation{New York}
(\byear{2006}).
\doiurl{10.1007/978-0-387-40065-5} .
\burl{http://dx.doi.org/10.1007/978-0-387-40065-5}
\end{bbook}
\endbibitem

\bibitem[\protect\citeauthoryear{Teschl}{2012}]{teschlbook}
\begin{bbook}
\bauthor{\bsnm{Teschl}, \binits{G.}}:
\bbtitle{Ordinary Differential Equations and Dynamical Systems}.
\bpublisher{American Mathematical Society},
\blocation{Providence}
(\byear{2012}).
\doiurl{10.1090/gsm/140} .
\burl{http://dx.doi.org/10.1090/gsm/140}
\end{bbook}
\endbibitem

\bibitem[\protect\citeauthoryear{Leoni}{2009}]{leoni2009first}
\begin{bbook}
\bauthor{\bsnm{Leoni}, \binits{G.}}:
\bbtitle{A First Course in Sobolev Spaces}.
\bpublisher{American Mathematical Society},
\blocation{Providence}
(\byear{2009}).
\doiurl{10.1090/gsm/105} .
\burl{http://dx.doi.org/10.1090/gsm/105}
\end{bbook}
\endbibitem

\bibitem[\protect\citeauthoryear{BULÍČEK et~al.}{2011}]{Bulek2011scalar}
\begin{barticle}
\bauthor{\bsnm{BULÍČEK}, \binits{M.}},
\bauthor{\bsnm{GWIAZDA}, \binits{P.}},
\bauthor{\bsnm{MÁLEK}, \binits{J.}},
\bauthor{\bsnm{ŚWIERCZEWSKA-GWIAZDA}, \binits{A.}}:
\batitle{On scalar hyperbolic conservation laws with a discontinuous flux}.
\bjtitle{Mathematical Models and Methods in Applied Sciences}
\bvolume{21}(\bissue{01}),
\bfpage{89}--\blpage{113}
(\byear{2011})
\doiurl{10.1142/s021820251100499x}
\end{barticle}
\endbibitem

\bibitem[\protect\citeauthoryear{Bulíček et~al.}{2017}]{Bulek2017unified}
\begin{barticle}
\bauthor{\bsnm{Bulíček}, \binits{M.}},
\bauthor{\bsnm{Gwiazda}, \binits{P.}},
\bauthor{\bsnm{Świerczewska-Gwiazda}, \binits{A.}}:
\batitle{On unified theory for scalar conservation laws with fluxes and sources
  discontinuous with respect to the unknown}.
\bjtitle{Journal of Differential Equations}
\bvolume{262}(\bissue{1}),
\bfpage{313}--\blpage{364}
(\byear{2017})
\doiurl{10.1016/j.jde.2016.09.020}
\end{barticle}
\endbibitem

\bibitem[\protect\citeauthoryear{Carrillo}{2003}]{Carrillo2003conservation}
\begin{barticle}
\bauthor{\bsnm{Carrillo}, \binits{J.}}:
\batitle{Conservation laws with discontinuous flux functions and boundary
  condition}.
\bjtitle{Journal of Evolution Equations}
\bvolume{3}(\bissue{2}),
\bfpage{283}--\blpage{301}
(\byear{2003})
\doiurl{10.1007/s00028-003-0095-x}
\end{barticle}
\endbibitem

\bibitem[\protect\citeauthoryear{Martin and
  Vovelle}{2008}]{Martin2008convergence}
\begin{barticle}
\bauthor{\bsnm{Martin}, \binits{S.}},
\bauthor{\bsnm{Vovelle}, \binits{J.}}:
\batitle{Convergence of implicit finite volume methods for scalar conservation
  laws with discontinuous flux function}.
\bjtitle{ESAIM: Mathematical Modelling and Numerical Analysis}
\bvolume{42}(\bissue{5}),
\bfpage{699}--\blpage{727}
(\byear{2008})
\doiurl{10.1051/m2an:2008023}
\end{barticle}
\endbibitem

\bibitem[\protect\citeauthoryear{Dias and
  Figueira}{2005}]{Dias2005approximation}
\begin{barticle}
\bauthor{\bsnm{Dias}, \binits{J.-P.}},
\bauthor{\bsnm{Figueira}, \binits{M.}}:
\batitle{On the approximation of the solutions of the riemann problem for a
  discontinuous conservation law}.
\bjtitle{Bulletin of the Brazilian Mathematical Society, New Series}
\bvolume{36}(\bissue{1}),
\bfpage{115}--\blpage{125}
(\byear{2005})
\doiurl{10.1007/s00574-005-0031-5}
\end{barticle}
\endbibitem

\bibitem[\protect\citeauthoryear{Dias and Figueira}{2004}]{Dias2004riemann}
\begin{barticle}
\bauthor{\bsnm{Dias}, \binits{J.-P.}},
\bauthor{\bsnm{Figueira}, \binits{M.}}:
\batitle{On the riemann problem for some discontinuous systems of conservation
  laws describing phase transitions}.
\bjtitle{Communications on Pure \& Applied Analysis}
\bvolume{3}(\bissue{1}),
\bfpage{53}--\blpage{58}
(\byear{2004})
\doiurl{10.3934/cpaa.2004.3.53}
\end{barticle}
\endbibitem

\bibitem[\protect\citeauthoryear{Rankine}{1870}]{Rankine1870thermodynamic}
\begin{barticle}
\bauthor{\bsnm{Rankine}, \binits{W.J.M.}}:
\batitle{Xv.\ on the thermodynamic theory of waves of finite longitudinal
  disturbance}.
\bjtitle{Philosophical Transactions of the Royal Society of London}
\bvolume{160},
\bfpage{277}--\blpage{288}
(\byear{1870})
\doiurl{10.1098/rstl.1870.0015}
\end{barticle}
\endbibitem

\bibitem[\protect\citeauthoryear{Hugoniot}{1887}]{hugoniot1887memoir}
\begin{botherref}
\oauthor{\bsnm{Hugoniot}, \binits{H.}}:
Memoir on the propagation of movements in bodies, especially perfect gases
  (first part).
J. de l’Ecole Polytechnique
\textbf{57}(3)
(1887)
\end{botherref}
\endbibitem

\bibitem[\protect\citeauthoryear{Keimer and
  Pflug}{2017}]{keimernonlocalbalance2017}
\begin{barticle}
\bauthor{\bsnm{Keimer}, \binits{A.}},
\bauthor{\bsnm{Pflug}, \binits{L.}}:
\batitle{Existence, uniqueness and regularity results on nonlocal balance
  laws}.
\bjtitle{Journal of Differential Equations}
\bvolume{263}(\bissue{7}),
\bfpage{4023}--\blpage{4069}
(\byear{2017})
\doiurl{10.1016/j.jde.2017.05.015}
\end{barticle}
\endbibitem

\bibitem[\protect\citeauthoryear{Bouchut and
  James}{1998}]{Bouchut1998onedimensional}
\begin{barticle}
\bauthor{\bsnm{Bouchut}, \binits{F.}},
\bauthor{\bsnm{James}, \binits{F.}}:
\batitle{One-dimensional transport equations with discontinuous coefficients}.
\bjtitle{Nonlinear Analysis: Theory, Methods \& Applications}
\bvolume{32}(\bissue{7}),
\bfpage{891}--\blpage{933}
(\byear{1998})
\doiurl{10.1016/s0362-546x(97)00536-1}
\end{barticle}
\endbibitem

\bibitem[\protect\citeauthoryear{Zhang and Liu}{2003}]{Zhang2003}
\begin{barticle}
\bauthor{\bsnm{Zhang}, \binits{P.}},
\bauthor{\bsnm{Liu}, \binits{R.-X.}}:
\batitle{Hyperbolic conservation laws with space-dependent flux: I.
  characteristics theory and riemann problem}.
\bjtitle{Journal of Computational and Applied Mathematics}
\bvolume{156}(\bissue{1}),
\bfpage{1}--\blpage{21}
(\byear{2003})
\doiurl{10.1016/s0377-0427(02)00880-4}
\end{barticle}
\endbibitem

\bibitem[\protect\citeauthoryear{LeVeque}{1992}]{LeVeque1992}
\begin{bbook}
\bauthor{\bsnm{LeVeque}, \binits{R.J.}}:
\bbtitle{Numerical Methods for Conservation Laws}.
\bpublisher{Birkh\"{a}user},
\blocation{Basel}
(\byear{1992}).
\doiurl{10.1007/978-3-0348-8629-1} .
\burl{http://dx.doi.org/10.1007/978-3-0348-8629-1}
\end{bbook}
\endbibitem

\bibitem[\protect\citeauthoryear{Bressan}{2000}]{Bressan2000}
\begin{bbook}
\bauthor{\bsnm{Bressan}, \binits{A.}}:
\bbtitle{Hyperbolic Systems of Conservation Laws: The One-Dimensional Cauchy
  Problem}.
\bpublisher{Oxford University PressOxford},
\blocation{Oxford}
(\byear{2000}).
\doiurl{10.1093/oso/9780198507000.001.0001} .
\burl{http://dx.doi.org/10.1093/oso/9780198507000.001.0001}
\end{bbook}
\endbibitem

\end{thebibliography}

\end{document}